\documentclass[11pt, a4paper, reqno]{amsart}
\usepackage{amsthm, amsmath, amsfonts, amssymb, appendix, dsfont, latexsym,mathrsfs,mathtools}
\usepackage{esint}
\usepackage{graphicx}
\usepackage{tikz}
\usepackage{bigints}

\makeatletter
\usepackage{delarray, a4, color}
\usepackage{euscript}
\usepackage[latin1]{inputenc}
\usepackage{enumitem}
\usepackage{hyperref}
\hypersetup{
colorlinks,
menucolor=black,
linkcolor=black,
citecolor=black,
urlcolor=blue
}
\usepackage{cleveref}
\usepackage{crossreftools}
\definecolor{darkgreen}{rgb}{0,0.5,0}
\definecolor{darkred}{cmyk}{0,1,1,0.4}

\theoremstyle{plain}
\newtheorem{theorem}{Theorem}[section]
\crefname{maintheorem}{Theorem}{Theorems} 
\newtheorem*{theorem*}{Theorem}
\newtheorem{lemma}{Lemma}[section]
\newtheorem{corollary}{Corollary}[section]
\newtheorem{proposition}{Proposition}[section]

\newtheorem{question}{Question}[section]

\theoremstyle{definition}

\newtheorem{example}{Example}
\newtheorem{remark}{Remark}[section]

\numberwithin{equation}{section}

\DeclareMathAlphabet{\mathpzc}{OT1}{pzc}{m}{it}

\def\C {\mathbb{C}}

\def\N {\mathbb{N}}

\def\R {\mathbb{R}}

\def\XXint#1#2#3{{\setbox0=\hbox{$#1{#2#3}{%
\int}$ }
\vcenter{\hbox{$#2#3$ }}\kern-.6\wd0}}

\newcommand{\bA}{\mathbf{A}}

\newcommand{\bx}{\mathbf{x}}
\newcommand{\by}{\mathbf{y}}
\newcommand{\bz}{\mathbf{z}}

\newcommand{\cE}{\mathcal{E}}

\newcommand{\cM}{\mathcal{M}}

\newcommand{\dd}{{\rm d}} 

\DeclareMathOperator{\curl}{\mathrm{curl}}

\newcommand{\loc}{\mathrm{loc}}

\usepackage{placeins}

\begin{document}
\title[Existence and uniqueness of solutions to Liouville equation]{Existence and uniqueness of solutions to Liouville equation}

\author[A. Ataei]{Alireza ATAEI}
\address{Alireza Ataei, Department of Mathematics, Uppsala University, Box 480, SE-751 06, Uppsala, Sweden}
\email{\url{alireza.ataei@math.uu.se}}


\allowdisplaybreaks

\maketitle

\setcounter{tocdepth}{2}

\begin{abstract}
 We prove some general results on the existence and uniqueness of solutions to the Liouville equation. Then, we discuss the sharpness and possible generalizations. Finally,
 we give several applications, arising in both mathematics and physics. 
\end{abstract}

\tableofcontents

\section{Introduction}
In this work, we study the existence and uniqueness of the solutions $\psi \in L^1_{\loc}(\R^2)$ to
\begin{equation}
\label{eq:mainequation}
\begin{aligned}
  &  -\Delta \psi =4\pi \beta V e^{\psi},\quad \textup{weakly in } \R^2,\\
  &  \int_{\R^2} V e^{\psi} =1,
\end{aligned}
\end{equation}
where $\beta \in \R \setminus 0$, $V$ is a locally integrable function, and $V e^{\psi} \in L^1(\R^2)$. Note that \eqref{eq:mainequation} is defined in the sense of distributions. The equation in \eqref{eq:mainequation} is known as the Liouville equation, and it has been considered in the literature due to its application in both mathematics and physics. We briefly mention some of the motivations here.

In differential geometry, the Liouville equation provides the conformally flat metrics on $\R^2$ with the Gaussian curvature $V$ and the total curvature $2 \pi\beta$; see \cite{A-38,ALN-24,A-86,CK-94,CK-00,CL-91,CL-93,CL-97,CL-98,CL-99,CL-00,CYZ-20,CN-91,CN-91(2),HY-98,Li-24,Lio-53,PT-01,M-84,M-85,N-82(2),P-898,S-73,WuZ-24}. If $\beta = 2$, by stereographic projection, one obtains all the possible metrics on the two dimensional unit sphere being conformally equivalent to the Euclidean unit sphere, which is known as the Nirenberg's problem; see \cite{A-21,CY-87,CY-88,CD-87,CL-95,CL-95(2),H-86,KW-74,S-05,XY-93}. Moreover, the problem is also generalized to metrics with conic singularities, where the metrics can vanish at finite points; see \cite{BDM-11,B-71,BT-02,B-71,CL-02,CL-03,CL-10,CL-15,E-04,E-20,E-21,P-905,T-08,T-89,T-91,WZ-18}. Initiated by a problem of Brezis and Merle; see \cite{BM-91}, the question of the blow-up property of solutions to the Liouville equation is considered in the analysis to understand the regularity properties of the solutions; see \cite{BM-07, BT-02,BT-07,BYZ-24,G-12,Li-99,LS-94,Wu-24,WZ-21,WZ-22,WZ-24,Z-06,Z-09}. In statistical mechanics, \eqref{eq:mainequation} is referred to as the mean field equation and has played a key role in understanding the two-dimensional turbulent flows; see \cite{BD-15,BJLY-18,CLMP-92,CLMP-95,CK-94,ES-93, ES-78,K-93, KL-97,Lin-00,O-49,PD-93,SO-90,TY-04,W-92}. Finally, \eqref{eq:mainequation} is used to study the ground states of the Chern-Simons-Schr\"odinger model used for the effective theory of many body anyons, which are two-dimensional quasi-particles, and the fractional quantum hall effect; see \cite{ALN-24,Dun-95,Dun-99,EHI-91,EHI-91(2),HY-98,HZ-09,IL-92,JW-90,JP-92,RS-20,SY-92,Z-92,T-08,ZHK-89}.

\subsection{Existence} Our first two results show the existence of solutions to \eqref{eq:mainequation}. To bring the results, for a function $V \in L^1_{\loc}(\R^2)$, we define $V^+ := \max(V,0),V^- = V^+ - V, D(\bx,r) \subset \R^2$ is the disk of radius $r>0$ centered at $\bx \in \R^2,$ and $\psi_0 \in C^{\infty}(\R^2)$ is a fixed radially symmetric function, satisfying $\psi_0(\bx) = 2 \beta \log|\bx|$ for $\bx \in \R^2 \setminus D(0,1)$. Moreover, define \begin{align*}
    \alpha(V) := \sup \left\{\alpha: \int_{\R^2 \setminus D(0,1)} |V(\bx)| \, |\bx|^{2\alpha} \, \dd \bx < \infty    \right\}.
\end{align*}

\begin{theorem}
\label{thm:maintheorem}
    Let $V \in L^1_{\loc}(\R^2)$ be a radially symmetric function which satisfies $V \geq C$ a.e. in $D(0,R_2) \setminus D(0,R_1)$ for some constants  $C>0, R_2> R_1>0$ and 
\begin{align}
    \label{eq:minimumcondionV}
\beta \geq -\alpha(V).
\end{align}
Then, the following holds:
\begin{enumerate}
    \item If $\beta>0$ and 
    \begin{equation}
    \begin{aligned}
    \label{eq:conditionondelta}
   &  \int_{\R^2 \setminus D(0,1)} V^-(\bx) |\bx|^{-2\beta} \, \dd \bx< \infty,\\
&\int_{D(0,1)} V^+(\bx) \,|\bx|^{-\beta-\delta} \, \dd \bx + \int_{\R^2 \setminus D(0,1)} V^+(\bx) |\bx|^{-\beta+\delta} \, \dd \bx< \infty,
\end{aligned}
\end{equation}
for some $\delta >0,$ then there exists a radially symmetric function $\psi \in H^1_{\loc}(\R^2) \cap C^1(\R^2 \setminus 0)$ such that 
    \begin{equation*}
    \begin{aligned}
&-\Delta \psi = 4 \pi \beta\, V e^{\psi}, \quad \textup{weakly in } \R^2, \\
&\int_{\R^2} |V| e^{\psi} < \infty, \,\int_{\R^2} V e^{\psi} =1, \, \int_{\R^2} |\nabla (\psi + \psi_0)|^2 < \infty,\\
& \lim_{\bx \to \infty} \frac{\psi(\bx)}{\log |\bx|} = -2\beta.
    \end{aligned}
    \end{equation*} Moreover, there exists an upper bound for 
    \begin{align}
    \label{eq:normofpsi}
        \int_{\R^2} |V| e^{\psi} + \|\nabla (\psi+\psi_0)\|_{L^2(\R^2)},\, 
    \end{align}
     which depends continuously on $\beta,R_1,R_2,C,\delta,$ and
    \begin{align*}
        \int_{D(0,1)} V^+(\bx) \, |\bx|^{-\beta-\delta} \, \dd \bx + \int_{\R^2 \setminus D(0,1)} V^+(\bx)\, |\bx|^{-\beta + \delta}\, \dd \bx + \int_{\R^2 \setminus D(0,1)} V^-(\bx) \, |\bx|^{-2\beta} \, \dd \bx.
    \end{align*}
    \item If $-\alpha(V) \leq \beta<0$ and $V$ is non-negative, then there exists a radially symmetric function $\psi \in H^{1}_{\loc}(\R^2) \cap C^1(\R^2 \setminus 0)$ such that 
    \begin{equation}
    \begin{aligned}
&-\Delta \psi = 4 \pi \beta\,\, V e^{\psi}, \quad \textup{weakly in } \R^2, \\
&\int_{\R^2} V e^{\psi} =1,\\
& \lim_{\bx \to \infty} \frac{\psi(\bx)}{\log |\bx|} = -2\beta.
    \end{aligned}
    \end{equation}    
\end{enumerate}

\end{theorem}
To compare our result with the literature, we mention the following remarks. 
\begin{remark}
   As far as we know, in the literature, it is assumed that $V$ is at least locally bounded to obtain a solution; see \cite{ALN-24,A-86,CK-94,CK-00,CL-91,CL-97,CL-98,CL-99,CL-00,CN-91,CN-91(2),CYZ-20,Li-24,Lio-53, Lin-00, M-84,M-85,N-82(2),S-05}. However, we have only the local integrability condition to derive a solution. The only exception we found is $V(\bx) = |\bx|^{-n}$ for $n<2$, which has been studied separately in \cite{PT-01}. 
\end{remark}

\begin{remark}
\label{rem:assymtoticradialsol}
    If $\psi \in L^1_{\loc}(\R^2) \cap C^1(\R^2 \setminus 0) $ is a radially symmetric solution to \eqref{eq:mainequation}, satisfying $V e^{\psi} \in L^1(\R^2)$, then by regularization and taking the integral of both sides of the equation in \eqref{eq:mainequation}, we derive that 
\begin{align*}
     r \partial_r \psi(r) = -2\beta \int_{D(0,r)} V  e^{\psi}, \quad r >0.
\end{align*}
Since $\int_{\R^2} V e^{\psi} =1$ and $\int_{\R^2} |V| e^{\psi} < \infty$, we conclude that
\begin{align*}
  \lim_{\bx \to \infty}  \frac{\psi(\bx)}{\log|\bx|} = -2\beta.
\end{align*}
Hence, for every $\delta >0$, we have 
\begin{align*}
    \psi(r) \geq -2 (\beta+\delta) \log r,
\end{align*}
 for $r>0$ large enough, which depends on $\delta$. In conclusion, by the assumption $\int_{\R^2} |V| e^{\psi }< \infty$, we conclude that 
\begin{align}
\label{eq:weakercondition}
    \int_{\R^2 \setminus D(0,1)} |V(\bx)| \, |\bx|^{-2(\beta +\delta)} \, \dd \bx <\infty,
\end{align}
for every $\delta >0.$ Hence, the condition \eqref{eq:minimumcondionV} is necessary for the existence of radially symmetric solutions. For the case of $\beta>0$, one can weaken the condition 
\begin{align*}
    \int_{\R^2 \setminus D(0,1)} V^-(\bx)  |\bx|^{-2\beta} \, \dd \bx < \infty,
\end{align*}
to \eqref{eq:weakercondition} if there exists a function $\varphi \in C^{\infty}([0,\infty))$ such that 
\begin{align}
\label{eq:keyconditionboundedenergy}
    \int_{1}^{\infty} V^- r^{-2\beta +1} e^{\varphi}   \,\dd r + \int_{1}^{\infty} |\partial_r \varphi|^2  r\, \dd r< \infty.
\end{align}
Indeed, we have used \eqref{eq:keyconditionboundedenergy} to derive that a Moser type energy functional is uniformly bounded from above. In particular, if 
\begin{align*}
  &  V^-(r) \leq C \left(r^2+1\right)^{\beta-1}, \quad \textup{for every } r \geq 0,\\
   & \int_{D(0,1)} V^+(\bx) \,|\bx|^{-\beta-\delta} \, \dd \bx + \int_{\R^2 \setminus D(0,1)} V^+(\bx) |\bx|^{-\beta+\delta} \, \dd \bx< \infty,
\end{align*}
for some positive constants $C,\delta,\beta$, then, by taking $\varphi := -2\log (\log (2+r))$ in \eqref{eq:keyconditionboundedenergy}, there exists a radially symmetric solution to \eqref{eq:mainequation}. This has been observed before if $V$ is asymptotically negative; see \cite[Thm. 1.1]{CL-99}.
 
\end{remark}

\begin{remark}
    The condition \eqref{eq:conditionondelta} implies that for $\beta \geq 2$, $V$ needs to vanish fast enough at the origin and infinity for the existence of a solution to \eqref{eq:mainequation}. In particular, if $V$ is continuous and $\beta \geq 2$, the condition \eqref{eq:conditionondelta} implies that $V$ should vanish at the origin. As far as we know, this phenomenon has been previously observed in the literature only in special cases; see \cite{CL-98,CL-99}, and is necessary for many general equations; see Corollary \ref{cor:existencepositiveassymdecay}.
     We warn the reader regarding \cite[Thm. 3]{M-85} which claims that the radially symmetric solutions to \eqref{eq:mainequation} exist always if $V$ is radially symmetric, decays fast enough, and $V(0)>0$. This issue was partially pointed out in \cite{CL-93,CL-98,CL-98(2),CL-99,W-92}, however, it is not well understood in the literature. By setting $V(r) = e^{-r^2}$ and using Corollary \ref{cor:existencepositiveassymdecay}, we get the solution to \eqref{eq:mainequation} exists if and only if $\beta <2$. Hence, even exponential decay is not sufficient to obtain a solution if $\beta \geq 2$. As we see in Example \ref{ex:blowupenergy}, the variational argument fails as the infimum of the Moser energy functional
      \begin{align}
     \label{eq:mosertrudingerenergyfunc}
         \cE_n[\varphi] := \frac{1}{2} \int_{D(0,n)} |\nabla \varphi|^2 - 4\pi \beta \log \left(\int_{D(0,n)} V e^{\varphi} \right), 
     \end{align}
     over radially symmetric functions $\varphi \in C^{\infty}_0(D(0,n))$ which satisfy $\int_{D(0,n)} V e^{\varphi}>0$, is $-\infty$ for any large enough $n$ if $V$ does not become zero fast enough at the origin.
    
\end{remark}

\begin{remark}
    We mention that for $\beta <0$ in Theorem \ref{thm:maintheorem}, under the extra assumption 
\begin{align}
\label{eq:strongerassumnegativebeta}
    \int_{\R^2 \setminus D(0,1)} V(\bx) \, |\bx|^{-2\beta}  \, \log|\bx| \, \dd \bx < \infty,
\end{align}
we can derive that 
\begin{align*}
    \int_{\R^2} |\nabla (\psi+ \psi_0)|^2 < \infty;
\end{align*}
see Remark \ref{rem:extracondiboundnegativebeta}. This has already been observed with the stronger assumption \begin{align*}
    \int_{\R^2} V(\bx) \, |\bx|^{-2\beta'} \, \dd \bx < \infty, 
\end{align*}
for some $\beta'< \beta$; see \cite{CK-00,CL-97,CL-99,CL-00,CN-91,CN-91(2),M-84,N-82(2),WuZ-24}.
\end{remark}

\begin{remark}
 The regularity of the solutions in Theorem \ref{thm:maintheorem} is optimal if one does not assume any other conditions on $V$, see Example \ref{ex:regularityexample}. If we assume that $V$ is non-negative, satisfying \eqref{eq:conditionondelta}, and $\psi$ is a solution to \eqref{eq:mainequation}, then $\psi$ is also locally $\alpha$-H\"older continuous for every $\alpha < \min(\beta + \delta,1);$ see Proposition \ref{prop:regularity}.

\end{remark}

\begin{remark}
     Our proof gives a constructive method to obtain the solutions and derives an explicit upper bound for \eqref{eq:normofpsi}, depending on the initial data, which, as far as we know, have not been previously studied. We use a simple variational approach that relies on the minimizers $\phi_n$ of the Moser energy functional $\cE_n[\varphi]$; see  \eqref{eq:mosertrudingerenergyfunc}, over radially symmetric functions $\varphi \in H^1_0(D(0,n))$ which satisfy $\int_{D(0,n)} V e^{\varphi}>0$. The key idea is to apply the condition \eqref{eq:conditionondelta} to prove that the functional  \eqref{eq:mosertrudingerenergyfunc} is coercive and bounded uniformly from above, where the bounds are independent of $n$. 
    Then, we show that $\phi_n - \log \left(\int_{D(0,n)} V e^{\phi_n} \right)$ converges as $n \to \infty.$ Although the minimizer of \eqref{eq:mosertrudingerenergyfunc} do not exist over $C^{\infty}_0(D(0,n))$ if $\beta \geq 2$, restricting to only radially symmetric functions resolves this issue if \eqref{eq:conditionondelta} holds; see Example \ref{ex:nonradialminimizing}. Moreover, unlike the previous works, we do not use a weighted Sobolev inequality, the ODE approach, Perron's method, complex analysis, or Statistical methods to derive the solutions; see \cite{ALN-24,A-86,CK-94,CK-00,CL-91,CL-93,CL-97,CL-98,CL-99,CL-00,CYZ-20,CN-91,CN-91(2),HY-98,Li-24,Lio-53,PT-01,M-84,M-85,N-82(2),P-898,S-73,WuZ-24}.
\end{remark}

 Now, by Theorem \ref{thm:maintheorem}, we obtain the following immediate corollary.
\begin{corollary}
\label{cor:assymptotic}
    Let $\beta >0, V \in L^1_{\loc}(\R^2)$ be a radially symmetric function satisfying $V \geq C$ a.e. in $D(0,R_2) \setminus D(0,R_1)$ for some constants  $C>0, R_2> R_1>0$. Assume that
    \begin{equation}
\begin{aligned}
\label{eq:asymptoticofV}
& V^-(r) \leq  C_0 r^{l_0}, \quad \textup{a.e. } r \geq 1,\\
&V^+(r) \leq C_1 r^{l_1}, \quad \textup{a.e. } r \leq 1,\\
&V^+(r) \leq C_2 r^{l_2}, \quad \textup{a.e. } r \geq 1, 
\end{aligned}
\end{equation}
where $C_0,C_1,C_2$ are positive and $l_1,l_2,l_3 \in \R$. Then, there exists a radially symmetric solution $\psi \in H^1_{\loc}(\R^2) \cap C^1(\R^2 \setminus 0)$ to \eqref{eq:mainequation} if $l_0 < 2\beta-2, l_1 > \beta -2, l_2< \beta-2$.

\end{corollary}

\begin{remark}
    We remark that for $\beta >2$, in Corollary \ref{cor:assymptotic}, we can take $l_2>0$ which allows $V$ to blow up at infinity. Moreover, if $0<\beta <2$, by taking $\beta-2<l_1 <0,l_2 < \beta-2$ we allow $V$ to blow up at zero. Furthermore, the bounds $l_1 > \beta-2, l_2 <\beta -2$ are sharp in the sense that there are counterexamples for $l_1 \leq \beta -2, l_2 < \beta-2$ or $l_1>\beta -2, l_2 \geq  \beta-2$. As an example, we can consider $ V(r) =r^{\beta-2} e^{-r^2}$ and use Corollary \ref{cor:existencepositiveassymdecay} to prove that there exists no solution to \eqref{eq:mainequation}. More surprisingly, there are examples of $V$, where $l_2 = 0$ and $l_0>0,l_1>0$ are arbitrary, for which there exist solutions to \eqref{eq:mainequation} but none are radially symmetric; see
    the constructed example in \cite[Thm. 3]{CL-95}. The assumption $l_0 < 2\beta-2$ can be improved to $l_0 \leq  2\beta -2$ by using the observation made in Remark \ref{rem:assymtoticradialsol}.
\end{remark}

By using Pokhozhaev identity and Theorem \ref{thm:maintheorem}, we prove the following result.

\begin{corollary}
\label{cor:existencepositiveassymdecay}
     Let $\beta \in \R \setminus 0$, $n > -2$, and $V \in C(\R^2) \cap C^1(\R^2 \setminus 0) \setminus 0$ be a non-negative and radially decreasing function such that $r V'(r) \in C([0,\infty))$. If $n > \beta -2$ and
     \begin{equation}
      \label{eq:uniformbound}
  \begin{aligned}
  \int_1^{\infty} r^{n+1-\beta+\delta} V(r) \, \dd r < \infty, \quad \textup{if } \beta >0,\\
  \int_1^{\infty} r^{n+1-2\beta} V(r) \, \dd r < \infty, \quad \textup{if } \beta <0,
\end{aligned}
\end{equation}
for some $\delta >0$, then there exists a radially symmetric $\psi \in C^1(\R^2 \setminus 0)$ which is $\alpha-$H\"older continuous in $\R^2$ for some $\alpha>0$ and satisfies
\begin{equation}
\label{eq:equationexistencepositvedecay}
\begin{aligned}
   & -\Delta \psi(\bx) = 4 \pi \beta |\bx|^n V(\bx) e^{\psi(\bx)},\quad \bx \in \R^2,\\
    & \int_{\R^2} |\bx|^n V(\bx) e^{\psi(\bx)}  \, \dd \bx= 1,\\
    & \lim_{\bx \to \infty} \frac{\psi(\bx)}{\log |\bx|} = -2\beta.
\end{aligned}
\end{equation}
 Moreover, if 
\begin{equation}
      \label{eq:uniformbound2}
  \begin{aligned}
  \int_1^{\infty} r^{n+1-2\beta+\delta} V(r) \, \dd r < \infty,
\end{aligned}
\end{equation}
for some $\delta >0$ and there exists a function $\psi \in L^1_{\loc}(\R^2)$ (not necessarily radially symmetric) which satisfies \eqref{eq:equationexistencepositvedecay}, then $n>\beta -2$ and
\begin{align}
\label{eq:indexformula}
\beta -2-n = \int_{\R^2} |\bx|^n e^{\psi(\bx)} \bx \cdot \nabla V(\bx) \, \dd \bx.
\end{align}
\end{corollary}

\begin{remark}
    We note that, by Remark \ref{rem:assymtoticradialsol}, every radial solution of \eqref{eq:mainequation} satisfies 
    \begin{align*}
         \lim_{\bx \to \infty} \frac{\psi(\bx)}{\log |\bx|} = -2\beta.
    \end{align*}
    However, the other direction does not necessarily hold; see Example \ref{ex:keyexample}.
\end{remark}
For the second result, for negative $\beta$ and general $V$ (not necessarily radially symmetric), we prove the following:

\begin{theorem}
\label{thm:strongestnegativebeta}
Let $\beta<0,$ and $V \in L^1_{\loc}(\R^2)$ be a non-negative function such that $V \geq C_1$ a.e. in a disk $D \subset \R^2$ for $C_1>0$ and $\alpha(V)>0$. If $\beta \geq -\alpha(V)$, then
there exists a solution $\psi_{\beta} \in L^{1}_{\loc}(\R^2)$ to
\begin{equation*}
\begin{aligned}
  &  -\Delta \psi_{\beta} =4\pi \beta V e^{\psi_{\beta}},\quad \textup{weakly in } \R^2,\\
  &  \int_{\R^2} V e^{\psi_{\beta}} =1.
\end{aligned}
\end{equation*}
Moreover, 
\begin{align}
    \label{eq:betadecreasing}
    \psi_{\beta_2} + \log|\beta_2| \leq \psi_{\beta_1} + \log|\beta_1|, \quad \textup{a.e. in } \R^2,
\end{align}
for every $\beta_2 \geq \beta_1 \geq -\alpha(V)$, and
\begin{equation}
 \begin{aligned}
\label{eq:assytheosolnegatigener}
 \psi_{\beta }(\bx) \leq -2\beta \log (|\bx|+1)+ \log |\beta| +C_2,\quad \textup{ for a.e. } \bx \in \R^2,
 \end{aligned}
 \end{equation}
 for every $\beta \geq -\alpha(V)$, where $C_2$ is a constant depending on $C_1,D$.
\end{theorem}

\begin{remark}
   As far as we know, this is the sharpest result so far for $\beta<0$ and generalizes the previous results in \cite{CK-00,CL-97,CL-98,CL-99,CL-00,CN-91,CN-91(2),CYZ-20,M-84,N-82(2),WuZ-24}, where $V$ is assumed to be at least bounded locally.

\end{remark}

\begin{remark}
   The proof of Theorem \ref{thm:strongestnegativebeta} uses the existence result of \cite[Thm. 2.10]{N-82}, which relies on Perron's method. This approach, unlike the proof of Theorem \ref{thm:maintheorem}, is non-constructive.
\end{remark}

\subsection{Uniqueness}
For the uniqueness of the solutions to \eqref{eq:mainequation}, we have two results regarding $\beta<0$ and $\beta>0$.

For $\beta <0$, we prove the following generalized comparison principle whereby we derive the uniqueness of solutions to \eqref{eq:mainequation} immediately if $F_1(x,r)=F_2(x,r)=-4\pi \beta\, V(x)\, e^r$.
\begin{theorem}
\label{thm:comparison}
    Let $F_1,F_2 \in C(\R^2 \times \R)$ be non-negative functions such that 
    \begin{align}
    \label{eq:comparisonassum}
        F_2(\bx,r_2) \geq F_1(\bx,r_1),    \end{align}
    for every $\bx \in \R^n, r_1,r_2 \in \R$ satisfying $r_2 \geq r_1$. Assume that 
$\psi_1,\psi_2 \in C^2(\R^2)$ satisfy
    \begin{align}
    \label{eq:equationscomparison}
 -\Delta \psi_2(\bx) +F_2(\bx,\psi_2(\bx)) \leq -\Delta \psi_1(\bx) + F_1(\bx,\psi_1(\bx)), \quad \textup{for every } \bx \in \R^2,
    \end{align} and
\begin{align}
\label{eq:aymptoticcondition}
   \liminf_{\bx \to \infty} \frac{\psi_1(\bx)-\psi_2(\bx)}{\log|\bx|}\geq 0.
\end{align}
 Then, either
\begin{align*}
    \psi_1 \geq \psi_2, \quad \textup{in } \R^2,
\end{align*}
or $\psi_2 >\psi_1$ in $\R^2$, $\psi_2- \psi_1$ is a constant, and
 \begin{align*}
    F_1(\bx,\psi_1(\bx)) = F_2(\bx,\psi_2(\bx)), \quad \bx \in \R^2.
\end{align*}
\end{theorem}

\begin{remark} An immediate conclusion of Theorem \ref{thm:comparison} is the uniqueness of solutions if $F_1 = F_2$, \eqref{eq:equationscomparison} is an equality, and 
\begin{align*}
    \lim_{x \to \infty} \frac{\psi_1(\bx)-\psi_2(\bx)}{\log|\bx|} = 0.
\end{align*}
Note that, for $F_1=F_2$, Theorem \ref{thm:comparison} is a well-known result if we replace $\R^2$ with bounded open subsets of $\R^2$; see \cite[Thm. 3.3]{CIL-92}. However, for unbounded domains such as $\R^2$,
 the asymptotic behavior 
 \begin{align}
\label{eq:assymptoticassump}
    \lim_{\bx \to \infty} \frac{\psi(\bx)}{\log|\bx|}= -2\beta,
\end{align}
 becomes essential. In fact, there might be infinitely many solutions to \eqref{eq:mainequation}, which have different asymptotic behaviors if $V$ decays fast enough; see Example \ref{ex:keyexample}. We mention that also if $F_1(\bx,r)$ has a continuous derivative in $r$ and 
\begin{align}
\label{eq:strongerassumcompa}
    \liminf_{x \to \infty} \frac{\psi_1(\bx)-\psi_2(\bx)}{\log|\bx|} >0,
\end{align}
 then one can use the well-known maximum principle to prove Theorem \ref{thm:comparison}. To see this, let $\Omega \subset \R^2$ be a bounded open subset, satisfying $\psi_1 > \psi_2$ on $\partial \Omega$, and define the functions $\psi := \psi_1 -\psi_2$ and $$\phi(\bx) := \frac{F_1(\bx,\psi_1(\bx)) -F_1(\bx,\psi_2(\bx))}{\psi_1(\bx)-\psi_2(\bx)}, \quad \textup{for every } \bx \in \R^2.$$ Then, by \eqref{eq:comparisonassum} and \eqref{eq:equationscomparison}, we have 
\begin{align*}
-\Delta  \psi + \phi \psi \geq 0, \quad \textup{in } \Omega,
\end{align*}
and
\begin{align*}
    \phi(\bx)\geq 0, \quad \textup{for every } \bx \in \R^2.
\end{align*}
 Hence, we can apply the maximum principle; see \cite[Ch. 6.4.1, Thm. 2; Ch. 6.4.2, Thm. 4]{E-10}, to get $\psi >0$ in $\Omega$. By \eqref{eq:strongerassumcompa}, we can take $\Omega$ large enough to derive that $\psi_1 >\psi_2$ in $\R^2.$
\end{remark}

\begin{remark}
  Let $V_1,V_2 \in C(\R^2) \setminus 0$ such that $V_2 \leq V_1 \leq 0$. Assume that $\psi_1,\psi_2 \in C^2(\R^2)$ are solutions to
    \begin{align*}
        -\Delta \psi_i =  V_i e^{\psi_i}, \quad \textup{in } \R^2,
    \end{align*}
for $i=1,2$ and satisfy
\begin{align*}
   \liminf_{\bx \to \infty} \frac{\psi_1(\bx)-\psi_2(\bx)}{\log|\bx|}\geq 0.
\end{align*}
 Then, by using Theorem \ref{thm:comparison} for $F_i(\bx,\psi_i) =- V_i(\bx) e^{\psi_i},$ we have either
\begin{align}
\label{eq:comparingtwosol}
    \psi_1 \geq \psi_2, \quad \textup{in } \R^2,
\end{align}
or $\psi_2 >\psi_1$ in $\R^2$, $\psi_2- \psi_1=C$ is a constant, and
 \begin{align*}
   V_2= V_1 e^{-C}, \quad \bx \in \R^2,
\end{align*}
where the second case is in contradiction with $V_1 \geq V_2$ and $V_1,V_2 \in C(\R^2) \setminus 0$. Hence, the only possibility is \eqref{eq:comparingtwosol}. 
\end{remark}

 Finally, for $\beta>0$, we prove the following theorem.
\begin{theorem}
\label{thm:uniquenessposistivbeta}
     Let $\beta>0$, $n \geq 0$, and $V \in C(\R^2) \cap C^1(\R^2 \setminus 0)$ be a non-negative and radially decreasing function. Assume that $r V'(r) \in C([0,\infty))$,
   \begin{align}
    \label{eq:uniformpointwisebound}
\sup_{r \geq 0} r^{n+2+\delta} V(r) < \infty, 
\end{align}
for some $\delta>0,$ and for every constant $0<c<n+2$ there exists $r_c >0$ satisfying 
\begin{equation}
\label{eq:unqiuenesscondition}
\begin{aligned}
c V(r) + r V'(r) \geq 0, \quad r \leq r_c,\\
c V(r) + r V'(r) \leq 0,\quad r \geq r_c.
\end{aligned}
\end{equation}
 Then, there exists at most one radially symmetric $\psi \in C^2(\R^2)$ such that
\begin{equation}
\label{eq:radialLiouvillesol}
\begin{aligned}
   & -\Delta \psi(\bx) = 4 \pi \beta |\bx|^n V(\bx) e^{\psi(\bx)},\quad \bx \in \R^2,\\
    & \int_{\R^2} |\bx|^n V(\bx) e^{\psi(\bx)} \, \dd \bx = 1.
\end{aligned}
\end{equation}
\end{theorem}
\begin{remark}
  In Theorem \ref{thm:uniquenessposistivbeta}, if $n=0$, then by \cite{CK-94} we see that all the solutions to \eqref{eq:mainequation}, which satisfy \eqref{eq:assymptoticassump}, are radially symmetric. In conclusion, by Theorem \ref{thm:uniquenessposistivbeta}, there exists a unique solution to \eqref{eq:mainequation} if $\beta>0, n=0$. Note that, by Corollary \ref{cor:existencepositiveassymdecay}, if $n=0$ we need $\beta<2$ for the existence of solutions to \eqref{eq:mainequation}. We could not verify whether all the solutions to \eqref{eq:radialLiouvillesol} are radially symmetric if $n>0$ and one assumes an asymptotic property such as \eqref{eq:assymptoticassump}. 
\end{remark}

\begin{remark}
\label{rem:infinitelymanyradialsol}
    If $V=1$ in Theorem \ref{thm:uniquenessposistivbeta}, then the uniqueness fails because of the conformal property of the equation. To see this, consider $\beta = 2n$ for $n>0$ and 
    \begin{align*}
        \psi_{\lambda,n}(\bx) = \log \left ( \frac{1 }{(\lambda^{2n} |\bx|^{2n} + 1)^2} \right ) + \log \left(\frac{n}{\pi}\right) + 2n \log(\lambda), \quad \bx \in \R^2,
    \end{align*}
for every $\lambda>0$. Then,
\begin{align*}
     & -\Delta \psi_{\lambda,n}(\bx) = 8 \pi n |\bx|^{2(n-1)}  e^{\psi_{\lambda,n}(\bx)},\quad \bx \in \R^2,\\
    & \int_{\R^2} |\bx|^{2(n-1)}  e^{\psi_{\lambda,n}(\bx)} \, \dd \bx = 1.
\end{align*}
We refer to \cite{CK-94,PT-01}, where they studied the above solutions and their conformal properties.
    
\end{remark}

\subsection{Open problems and discussions}

We present three open problems related to \eqref{eq:mainequation}.
\begin{question}
    Is there a general existence result for the equation \eqref{eq:mainequation} in the case of functions $V$ which are not radially symmetric and $\beta \geq 2$?
\end{question}
 We mention that our approach fails as we rely on the properties of the symmetric functions for finding the minimizes of the energy functionals; see Example \ref{ex:nonradialminimizing}. Furthermore, there are counterexamples where $V$ satisfies the conditions \eqref{eq:minimumcondionV} and \eqref{eq:conditionondelta} for $\beta>0$, but there are no solutions to \eqref{eq:mainequation}. This can be seen as follows: Let $\beta =2$. By stereographic projection and conformal transformations, it is sufficient to find a Gaussian curvature $K$ on the two-dimensional unit sphere $(S^2,g)$, where $K$ and $\nabla^n K$ vanish at two points for any $i \leq n$ and large enough $n$, with the metric $g$ which is not conformally equivalent to the two-dimensional unit sphere with the Euclidean metric. To do so, we consider a radially symmetric $K$ which is non-positive in the southern hemisphere, $K$ and $\nabla^i K$ vanishes simultaneously at least at two points for any $i \leq n$ and large enough $n$, and $K$ is increasing and non-negative in the northern hemisphere. Then, by \cite[Thm. 1]{CL-95}, the two-dimensional unit sphere $(S^2,g)$ with the Gaussian curvature $K$ is not conformally equivalent to the Euclidean unit sphere. In conclusion, understanding the behavior of the function $V$ at the origin and infinity is insufficient to study the existence of the solutions to \eqref{eq:mainequation}.

\begin{question}
    What general conditions on $V$ imply that there exists a unique solution to \eqref{eq:mainequation} for $\beta>0$? 
\end{question}
In Theorem \ref{thm:uniquenessposistivbeta}, we have proved the uniqueness of radially symmetric solutions to \eqref{eq:mainequation} under the symmetry condition and asymptotic decay on $V$. However, it is not clear if there does not exist a non-radially symmetric solution if $n>0$. We note that the uniqueness fails for $\beta=2$ if one considers a non-radially symmetric $V$ even if we assume \eqref{eq:minimumcondionV} and \eqref{eq:conditionondelta}. In fact, by \cite[Thm. 2]{XY-93}, \cite[Thm. 3]{CL-95}, and stereographic projection, there exists a bounded radially symmetric function $V$, which is negative around the origin, and \eqref{eq:mainequation} has a non-radially symmetric solution. Hence, by rotation and conformal transformations, we obtain different solutions to \eqref{eq:mainequation} for a bounded non-radially symmetric Gaussian curvature, which is negative around the origin and infinity. In particular, they satisfy \eqref{eq:minimumcondionV} and \eqref{eq:conditionondelta}.

 \begin{question}
  Does a small perturbation of $V$ influence the existence or uniqueness?    
 \end{question}
 The answer to this question is positive in general, but there is a question of how it is influenced by the structure of $V$. To give some examples, we can consider $\beta=2$. Then, by Corollary \ref{cor:existencepositiveassymdecay}, there exists a solution to \eqref{eq:mainequation} for $V = r^{\delta}e^{-r^2}$ if $\delta>0$ but there are no solutions to \eqref{eq:mainequation} for $V=e^{-r^2} .$ Moreover, by Remark \ref{rem:infinitelymanyradialsol}, there exist infinitely many radially symmetric solutions to \eqref{eq:mainequation} for $V = 1$, however, there exists no solution to \eqref{eq:mainequation} if $V = e^{-\delta r^2}$ and $\delta>0$; see Corollary \ref{cor:existencepositiveassymdecay}. The example in which the answer is unclear is $V= |f|^2 e^{ \frac{-B |\bx|^2}{2} }$, where $B>0$ and $f$ is a complex polynomial. As we will see in Section \ref{sec:Quantum mechanincs}, this example appears in the study of nonlinear Landau levels of the Chern-Simons-Schr\"odinger model with an external magnetic field of strength $B$. We will prove the existence of the solutions if $f$ has only roots at the origin. To answer the general case, it is relevant to consider the perturbation of the symmetric case as follows: Let  $B>0$, $f_0(z) = z^n$, and $g_{\lambda}(z) = f_0(z) + \lambda g(z)$, where $\lambda>0,z \in \C, n \geq 0$, where $g$ is a complex polynomial of degree less than $f_0$. Then, is it possible to perturb the radially symmetric solutions to \eqref{eq:mainequation} for $V= |f_0|^2 e^{-B |\bx|^2}$ to obtain a solution to \eqref{eq:mainequation} for $V= |g_{\lambda}|^2 e^{-B |\bx|^2}$ if $\lambda$ is small enough?

\section{Preliminaries}

In this section, we assume that $V \in L^1_{\loc}(\R^2)$ satisfies  $V \geq C$ in $D(0,R_2) \setminus D(0,R_1)$ for some constants  $C>0, R_2> R_1>0$. For every open subset $\Omega \subset \R^2$ and $\alpha>0$,  the space $C^{0,\alpha}(\Omega)$ includes all the $\alpha$-H\"older continuous functions on $\Omega$, and $C_{\loc}^{0,\alpha}(\Omega)$ denotes the locally $\alpha$-H\"older continuous functions on $\Omega$. We use the notation $(r,\theta)$ as the polar coordinate on $\R^2$, and $\varphi(\bx) = \varphi(|\bx|)$ for a.e. $\bx \in \R^2$ and every radially symmetric function $\varphi \in L^1_{\loc}(\R^2)$.

\begin{proposition}
\label{prop:logbound}
    Let $R>0$ and $\varphi \in H^1(D(0,R))$ be a radially symmetric function. Then, $\varphi \in W^{1,\infty}_{\loc}(D(0,R) \setminus 0)$ and
\begin{align*}
    |\varphi(r)- \varphi(s)|^2 \leq \frac{\log\left(\frac{r}{s}\right)}{2\pi} \int_{D(0,R)} |\nabla \varphi|^2, \quad \textup{for every } 0<s\leq r \leq R. 
\end{align*}
Moreover, if $\Delta \varphi \in L^1(D(0,R))$, then $\varphi \in C^1(D(0,R) \setminus 0)$ and satisfies
\begin{align*}
   |\bx|\, |\nabla \varphi(\bx)| \leq \frac{1}{2\pi} \int_{D} |\Delta \varphi|,
\end{align*}
for every $\bx \in D(0,R) \setminus 0$.
    
\end{proposition}

\begin{proof}
The proof is a standard application of the Cauchy-Schwarz inequality. For the convenience of the reader, we bring it here. Since $\partial_r \varphi \in L^2_{\loc}(D(0,R) \setminus 0)$, by the Sobolev embedding theorem, we obtain $\varphi \in C^{0,\alpha}_{\loc}(D(0,R) \setminus 0)$ for some $\alpha>0$. Moreover, by Cauchy-Schwarz inequality, we imply
\begin{equation}
\label{eq:lipschitzchauchySchwarzbound}
\begin{aligned}
   |\varphi(r)-\varphi(s)|^2 &\leq  \left(\int_s^r |\nabla \varphi(t)| \, \dd t \right)^2\\ &\leq \frac{1}{2\pi} \int_s^r \frac{1}{t} \, \dd t \int_{D(0,R)} |\nabla \varphi|^2 \\&\leq \frac{\log\left(\frac{r}{s}\right)}{2\pi} \int_{D(0,R)} |\nabla \varphi|^2,
   \end{aligned}
   \end{equation}
for  $0<s\leq r\leq R.$ Hence, $\varphi \in W^{1,\infty}_{\loc}(\R^2 \setminus 0)$. Finally, by assuming $\Delta \varphi \in L^1(D(0,R))$, taking the integration of the regularization of $\Delta \varphi$, and using the radially symmetric property of $\varphi$, we arrive at $\varphi \in C^1(D \setminus 0)$ and
\begin{align*}
     |\bx| \, |\nabla \varphi(\bx)| \leq  
 \frac{\int_{D} |\Delta \varphi|}{2\pi },
\end{align*}
for every $\bx \in D \setminus 0.$ 
\end{proof}

\begin{proposition}
\label{prop:regularity}
    Let $D \subset \R^2$ be a finite disk centered at the origin, $V$ be a radially symmetric function such that 
    \begin{align}
\label{eq:strongnegativeintegral}
    |V(\bx)|\,  |\bx|^{-\delta} \in L^1(D),
\end{align}
   for some $\delta>0$, and $\psi \in H^1(D)$ be a radially symmetric distributional solution of 
\begin{align}
\label{eq:dirstribequpsi}
    -\Delta \psi = V e^{\psi},\quad \textup{in } D,
\end{align}
where $V e^{\phi} \in L^1(D).$ Then, $\psi \in C^{0,\alpha}_{\loc}(D)$ for every $\alpha$ satisfying
\begin{align*}
    \alpha = 1, \quad \textup{if } &\delta >1,\\
    0<\alpha < \delta, \quad \textup{if } &\delta \leq 1.
\end{align*}

    \end{proposition}

\begin{proof}
   By Proposition \ref{prop:logbound}, we arrive at $\psi \in C^1(D \setminus 0)$ and
\begin{align}
    \label{eq:gradboundpsiprop}
     |\bx| \, |\nabla \psi(\bx)| \leq  
 \frac{\int_{D} |V|\, e^{\psi}}{2\pi },
\end{align}
for every $\bx \in D \setminus 0.$ Let $R$ be the radius of $D$. Then, by Proposition \ref{prop:logbound} and \eqref{eq:gradboundpsiprop}, we derive that 
\begin{align*}
   & |\nabla \psi(\bx)| \leq  \frac{\int_{D(0,|\bx|)} |V|\, e^{\psi}}{2\pi }\, |\bx|^{-1}  \\& \leq  \frac{\int_{D} |V(\by)|  \, e^{\psi(\by)}\, |\by|^{\varepsilon-\delta} \, \dd \by}{2\pi } 
 \,|\bx|^{\delta-\varepsilon-1} 
    \\ &\leq   \frac{ e^{\psi(\bz) + \frac{\int_{D} |\nabla \psi|^2}{8\pi \varepsilon}} \int_{D} |V(\by)|  \, e^{\frac{2\pi \varepsilon (\psi(\by)-\psi(\bz))^2}{\int_{D} |\nabla \psi|^2} }\, |\by|^{\varepsilon-\delta} \, \dd \by}{2\pi }\,  |\bx|^{\delta-\varepsilon-1} 
   \\& \leq   \frac{ e^{\psi(\bz) + \frac{\int_{D} |\nabla \psi|^2}{8\pi \varepsilon}} \int_{D} |V(\by)|  \,  \, |\by|^{-\delta} \, \dd \by}{2\pi } \left(1+\left(\frac{R}{|\bz|} \right)^{2\varepsilon}\right)\, |\bx|^{\delta-\varepsilon-1}, 
\end{align*}
for every $\bz \in D \setminus 0$ and $0<\varepsilon<\delta.$ Finally, by taking the integral of both sides above, we complete the proof.

\end{proof}

\begin{lemma}
\label{lem:symmetricliouvillenegativebeta}
    Let $\beta <0$, $V$ be radially symmetric and non-negative, and $D \subset \R^2$ be a disk centered at the origin such that $D(0,R_2) \subset D$. Then, there exists a radially symmetric function $\phi \in H^1_0(D) \cap C^1(D\setminus 0)$ which satisfies 
    \begin{align*}
       -\Delta \phi = 4 \pi \beta \frac{V e^{\phi}}{\int_{D} V e^{\phi} },
    \end{align*}
weakly in $D$, where $\int_{D} V e^{\phi} < \infty$. Moreover, $\phi$ is a radially increasing function which satisfies 
\begin{align*}
   |\bx|\, |\nabla \phi(\bx)| \leq -2\beta,
\end{align*}
for every $\bx \in D \setminus 0$.
\end{lemma}

\begin{proof}
    The idea of the proof is to first consider the minimization of the energy functional 
    \begin{align*}
      \mathcal{E}[\varphi] =  \frac{1}{2} \int_{D} |\nabla \varphi|^2 - 4 \pi \beta \log \left ( \int_{D} V e^{\varphi} \right ),
    \end{align*}
 over $ \cM$ which consists of radially symmetric functions $\varphi \in H^1_0(D)$ such that $\int_{D} V e^{\varphi}< \infty$. To show that there exists a minimizer for $\mathcal{E}$ over $ \cM$, we consider $\lambda > \max\left(-\frac{\beta}{2},1\right)$, $\Omega:= D(0,R_2) \setminus D(0,R_1)$, and $\varphi = \lambda \psi$, where $\varphi \in  \cM \setminus 0$. Then,

\begin{align*}
 &\log \left ( \int_{D} V e^{\varphi} \right ) \geq 
    \log(C) + \log \left ( \int_{\Omega} e^{\lambda \psi} \right )
    \\ &\geq  \log\left(\frac{C}{|\Omega|^{\lambda-1}}\right) +\lambda \log \left ( \int_{\Omega} e^{ \psi} \right ),
\end{align*}
where we used H\"older's inequality on the last inequality. Hence,
\begin{equation}
\label{eq:energylowerb1}
\begin{aligned}
   \mathcal{E}[\varphi] &\geq \frac{\lambda^2}{2} \int_{D} |\nabla \psi|^2 -4 \pi \beta \left ( \log\left(\frac{C}{|\Omega|^{\lambda-1}}\right) +\lambda \log \left ( \int_{\Omega} e^{ \psi} \right ) \right)
    \\ &= -4\pi \beta \lambda \left(\left(-\frac{\lambda}{8\pi \beta}- \frac{1}{16\pi}\right) \int_{D} |\nabla \psi|^2 \right)\\& -4\pi \beta \lambda \left(\frac{1}{16\pi } \int_{D} |\nabla \psi|^2 +  \log \left ( \int_{\Omega} e^{ \psi} \right)\right)- 4\pi \beta \log\left(\frac{C}{|\Omega|^{\lambda-1}}\right).
\end{aligned}
\end{equation}
Now, by Cauchy-Schwarz inequality and Trudinger-Moser inequality; see \cite{T-67, M-71}, we have
\begin{equation}
\label{eq:moserbound}
    \begin{aligned} C_D
 &e^{\left(\frac{1}{16 \pi } \int_{D} |\nabla \psi|^2 + \log \left ( \int_{\Omega} e^{ \psi} \right)\right)} \\&\geq e^{\left(\frac{1}{16 \pi } \int_{D} |\nabla \psi|^2 + \log \left ( \int_{\Omega} e^{ \psi} \right)\right)} \int_D e^{\left(\frac{4\pi |\psi|^2}{\int_D |\nabla \psi|^2}\right)}
 \\ & \geq \int_{D} e^{-\psi} \int_{\Omega} e^{\psi} \geq  |\Omega|^2,
    \end{aligned}
\end{equation}
for a constant $C_D$ which depends on $|D|$. Hence, by \eqref{eq:energylowerb1} and \eqref{eq:moserbound}, we obtain that 
\begin{align}
\label{eq:energyb2}
    \mathcal{E}[\varphi]\geq C_1 \int_{D} |\nabla \varphi|^2-C_2 ,
 \end{align}
where $C_1,C_2\in \R$ are positive constants which depend on $D,\Omega,C,\lambda,\beta.$ Now, let $\varphi_n \in  \cM$ be a minimizing sequence for $\mathcal{E}$. Then, by \eqref{eq:energyb2}, we obtain that both $\int_D |\nabla \varphi_n|^2$ and $ \left|\log\left(\int_D V e^{\varphi_n}\right)\right|$ are uniformly bounded. Hence, by Rellich-Kondrakov theorem and Fatou's lemma, we conclude that $\varphi_n$ converges pointwise and in $L^p(D)$ to $\phi \in  \cM$ for every $p>0$, up to a subsequence, and 
\begin{align*}
    \mathcal{E}[\phi] \leq \liminf_{n \to \infty} \mathcal{E}[\varphi_n].
\end{align*}
Finally, since $\phi$ is a critical point for $\mathcal{E}$ over $ \cM$, we obtain 
  \begin{align}
  \label{eq:equationphisymmetricnegative}
       -\Delta \phi = 4 \pi \beta \frac{V e^{\phi}}{\int_{D} V e^{\phi} },
    \end{align}
weakly in $D$. Moreover, Proposition \ref{prop:logbound}, we arrive at $\phi \in C^1(D \setminus 0)$ and
\begin{align*}
     |\bx| \, |\nabla \phi(\bx)| = 4 \pi \beta 
 \frac{\int_{D} V e^{\phi}}{\int_{D} V e^{\phi} } = 4 \pi \beta,
\end{align*}
for every $\bx \in D \setminus 0,$ which completes the proof.

\end{proof}

\begin{remark}
    The same proof could have been used to obtain a more general result as follows: Let $\beta<0, \Omega \subset \R^2$ be a bounded open subset such that $D(0,R_2) \subset \Omega$, and $V$ be a non-negative function which is not necessarily radially symmetric. Then, there exists $\phi \in H^1_0(\Omega)$ such that
\begin{align*}
    -\Delta \phi = 4 \pi \beta \frac{V e^{\phi}}{\int_{\Omega} V e^{\phi}}, \quad \textup{weakly in } \R^2.
\end{align*}

\end{remark}

\begin{lemma}
\label{lem:degreeformula}
Let $V \in L^1_{\loc}(\R^2), \psi \in L^1_{\loc}(\R^2)$ such that $V e^{\psi} \in L^1 \cap L^p(\R^2)$ for some $p>1.$ Assume that
\begin{align*}
    -\Delta \psi = V e^{\psi},
\end{align*}
weakly in $\R^2$ and either of the following conditions holds:\\
$(i).$ The function $\psi$ is bounded from above and $\int_{\R^2} V e^{\psi} \geq 0$.\\
$(ii).$ The function $\psi$ is bounded from below and $\int_{\R^2} V e^{\psi} \leq 0$.\\\\
Then,
\begin{align*}
  \lim_{\bx \to \infty}  \frac{\psi(\bx)}{\log|\bx|} = -\frac{\int_{\R^2} V e^{\psi}}{2\pi}.
\end{align*}
Moreover, if we assume the additional condition that 
\begin{equation}
\begin{aligned}
\label{eq:assymptoticdecreaseofVlog}
&\int_{\R^2} |V(\by)|\, e^{\psi(\by)} \log (|\by|+1) \, \dd \by < \infty,
\end{aligned}
\end{equation}
    then there is a constant $C>0$ such that
    \begin{align*}
    - \frac{\int_{\R^2} V e^{\psi}}{2\pi} \log( |\bx| +1) -C  \leq \psi(\bx) \leq -\frac{\int_{\R^2} V e^{\psi}}{2\pi} \log (|\bx| +1) + C, \quad\textup{for a.e. } \bx \in \R^2.
    \end{align*}
\end{lemma}
\begin{proof}
    We prove the case that $\psi$ satisfies $(i)$ and the other case follows from the same argument. The argument is similar to \cite[Thm. 1]{CL-93} and \cite[Lem. 3.15]{ALN-24}, and for the convenience of the reader, we bring it here. Since $V e^{\psi} \in L^1(\R^2) \cap L^p(\R^2)$ for some $p>1$, by  \cite[Lem. 3.6, Lem. 3.7, Lem. 3.15]{ALN-24}, we can define 
\begin{align*}
   \phi(\bx) :=-\frac{1}{2\pi} \int_{\R^2}  (\log|\bx-\by| - \log (|\by|+1) ) \, V(\by) e^{\psi(\by)} \, \dd \by, \quad \textup{for every } \bx \in \R^2,
\end{align*}
which satisfies
\begin{align*}
\phi &\in  W^{2,p}_{\loc}(\R^2),\\
-\Delta  \phi &= V e^{\psi}, \quad \textup{weakly in } \R^2,\\
  \lim_{\bx \to \infty}  \frac{\phi(\bx)}{\log|\bx|} &= -\frac{\int_{\R^2} V e^{\psi}}{2\pi}.
\end{align*}
    Hence, $\psi - \phi$ is harmonic and bounded from above by $C \log |\bx|$ for some large enough constant $C>0$. In conclusion, $\psi-\phi$ is a constant and 
    \begin{align*}
         \lim_{\bx \to \infty}  \frac{\psi(\bx)}{\log|\bx|} = \lim_{\bx \to \infty}  \frac{\phi(\bx)}{\log|\bx|} =-\frac{\int_{\R^2} V e^{\psi}}{2\pi}.
    \end{align*}
The rest of the proof follows the same as the proof of \cite[Thm. 1]{CL-93} and we avoid repetition.

 \end{proof}

\section{Proof of existence of solutions to Liouville equation}
In this section, we prove Theorem \ref{thm:maintheorem} and Theorem \ref{thm:strongestnegativebeta}.
\begin{proof}[Proof of Theorem \ref{thm:maintheorem}]
    First, we consider the case that $-\alpha(V)<\beta <0$ and $V$ is non-negative. Then, by definition of $\alpha(V)$ and $V \in L^1_{\loc}(\R^2)$, we have 
    \begin{align}
    \label{eq:keyassumnegativebeta}
    \int_{\R^2} V(\bx) \, \left(1+|\bx|^{-2\beta}\right) \, \dd \bx < \infty.
    \end{align}
     Without loss of generality, by scaling, we can assume that $ V \geq C$ a.e. in $D(0,2) \setminus D(0,2-\delta)$ for every small enough $\delta>0.$ Define the radially symmetric regularization
\begin{align*}
    V_{\varepsilon} = \varphi_{\varepsilon} \ast \left(V 1_{D(0,1/\varepsilon)} \right) \in C^{\infty}_c(\R^2),
\end{align*}
for every $0<\varepsilon<1$ small enough such that $D(0,3) \setminus D(0,1) \subset D(0,1/\varepsilon)$, where
\begin{align*}
    \varphi_{\varepsilon}(\bx) := \frac{\varphi(\bx/\varepsilon)}{\varepsilon^2}, \quad \textup{for every } \bx \in \R^2,
\end{align*} and  $\varphi \in C^{\infty}_c(\R^2)$ is a non-negative radially symmetric function which satisfies 
\begin{align*}
    \int_{\R^2} \varphi =1.
\end{align*}
Then, $V_{\varepsilon} $ converges to $V$ pointwise and in $L^1(\R^2)$, up to a subsequence, and
\begin{equation}
\label{eq:conditiononapproxforV1}
\begin{aligned}
  \lim_{\varepsilon \to 0}  \int_{\R^2} |V(\bx)-V_{\varepsilon}(\bx)|\, |\bx|^{-2\beta}\, \dd \bx =0,\\
   V_{\varepsilon} \geq C_1, \quad \textup{in } D(0,2) \setminus D(0,2-\delta),
\end{aligned}
\end{equation}
for every $\delta>0$ small enough, which does not depend on $\varepsilon.$ Indeed, 
\begin{align*}
     |V(\bx)| \, \left(\varphi_{\varepsilon} \ast |\by|^{-2\beta}\right)(\bx)  \leq 2  |V(\bx)| \, (|\bx|^{-2\beta} +\varepsilon^{-2\beta}), \quad \textup{for every } \bx \in \R^2,
\end{align*}
and therefore, by \eqref{eq:keyassumnegativebeta} and Lebesgue's dominated convergence, we get
\begin{align*}
   \lim_{\varepsilon \to 0}\int_{\R^2}  \varphi_{\varepsilon} \ast|V| (\bx) \, |\bx|^{-2\beta} \, \dd \bx =\lim_{\varepsilon \to 0}\int_{\R^2}  |V(\bx)| \,\left(\varphi_{\varepsilon} \ast   |\by|^{-2\beta}\right)(\bx) \, \dd \bx 
   = \int_{\R^2}  |V(\bx)| \, |\bx|^{-2\beta} \, \dd \bx. 
\end{align*}
Hence, by using $|V_{\varepsilon}| \leq \varphi_{\varepsilon} \ast|V|$ and Lebesgue's dominated convergence again, we obtain
\begin{align*}
    \lim_{\varepsilon \to 0}  \int_{\R^2} |V(\bx)-V_{\varepsilon}(\bx)|\, |\bx|^{-2\beta}\, \dd \bx =0.
\end{align*}
     Let $n>2$ be an integer. Then, by \eqref{eq:keyassumnegativebeta} and Lemma \ref{lem:symmetricliouvillenegativebeta}, there exists a sequence of radially symmetric and radially increasing functions $\phi_{n,\varepsilon} \in H^1_0(D(0,n))$ which satisfy
\begin{equation}
\begin{aligned}
\label{eq:equationforphi}
- \Delta \phi_{n,\varepsilon} &= \frac{4 \pi \beta \, V_{\varepsilon} e^{\phi_{n,\varepsilon}}}{\int_{D(0,n)} V_{\varepsilon} e^{\phi_{n,\varepsilon}}}, \quad \textup{in } D(0,n),
\end{aligned}
\end{equation}
and 
\begin{align}
\label{eq:gradphibound}
       |\bx| \, |\nabla \phi_{n,\varepsilon}(\bx)| &\leq -2 \beta, \quad \textup{for every } \bx \in \R^2.
\end{align}
By integrating the above inequality, we derive that
\begin{align}
\label{eq:lipschitzbound}
    |\phi_{n,\varepsilon}(\bx) - \phi_{n,\varepsilon}(\by)| \leq -2\beta \left | \log \left ( \frac{|\bx|}{|\by|} \right ) \right |,
\end{align}
for every $\bx, \by \in D(0,n) \setminus 0.$ Define the function
\begin{align}
    \psi_{n,\varepsilon} := \phi_{n,\varepsilon} - \log\left( \int_{D(0,n)} V e^{\phi_{n,\varepsilon}} \right).
\end{align}
Then,
\begin{equation}
\label{eq:upperbound}
\begin{aligned}
     \psi_{n,\varepsilon}(\bx) &= -\log\left( \int_{D(0,n)} V_{\varepsilon}(\by) e^{\phi_{n,\varepsilon}(\by)-\phi_{n,\varepsilon}(\bx)} \, \dd \by \right) \\ &\leq -2\beta \log |\bx| - \log\left( |\bx|^{-2\beta}\int_{D(0,n) \setminus D(0, |\bx|)} V_{\varepsilon}(\by) \, \dd \by +   \int_{D(0, |\bx|)} V_{\varepsilon}(\by) |\by|^{-2\beta} \, \dd \by\right),
\end{aligned}
\end{equation}
for every $\bx \in D(0,n) \setminus 0$, where we used the property that $\phi_n$ is radially increasing and satisfies \eqref{eq:lipschitzbound}. Likewise, we obtain
\begin{equation}
\label{eq:lowerbound}
\begin{aligned}
    \psi_{n,\varepsilon}(\bx) &= -\log\left( \int_{D(0,n)} V_{\varepsilon}(\by) e^{\phi_{n,\varepsilon}(\by)-\phi_{n,\varepsilon}(\bx)} \, \dd \by \right)
    \\ &\geq -\log\left( \int_{D(0,|\bx|)} V_{\varepsilon}(\by) \, \dd \by  + |\bx|^{2\beta}  \int_{D(0,n) \setminus D(0,|\bx|)} V_{\varepsilon}(\by) |\by|^{-2\beta} \, \dd \by\right)
    \\ & \geq- 2 \beta  \log |\bx| -\log\left( |\bx|^{-2\beta} \int_{\R^2} V_{\varepsilon}(\by) \, \dd \by  +  \int_{\R^2} V_{\varepsilon}(\by) |\by|^{-2\beta} \, \dd \by\right),
\end{aligned}
\end{equation}
for every $\bx \in D(0,n) \setminus 0$. Hence, $\|\psi_{n,\varepsilon}\|_{L^p(D)}$ is uniformly bounded for every $p \geq 1$ and every disk $D \subset D(0,n)$. Moreover, by \eqref{eq:gradphibound}, we get $\|\nabla \psi_{n,\varepsilon}\|_{L^q(D)}$ is uniformly bounded for every $0<q<2$ and $D$ as before. In conclusion, by the Rellich-Kondrakov compactness theorem, we have $\psi_{n,\varepsilon}$, up to a subsequence, converges pointwise and in $L^2_{\loc}(\R^2)$ to a radially symmetric function $\psi_{\varepsilon} \in W^{1,q}_{\loc}(\R^2)$ for every $0<q<2$. By \eqref{eq:upperbound}, we have 
\begin{align*}
    0 \leq V_{\varepsilon}(\bx) e^{\psi_{n,\varepsilon}(\bx)} \leq  \frac{ V_{\varepsilon}(\bx) |\bx|^{-2\beta}}{\int_{D(0,2-\delta)} V_{\varepsilon}(\by) |\by|^{-2\beta} \, \dd \by },
\end{align*}
for every $\bx \in D(0,n) \setminus D(0,2-\delta)$ and 
\begin{align*}
   0\leq  V_{\varepsilon}(\bx) e^{\psi_{n,\varepsilon}(\bx)} \leq \frac{ V_{\varepsilon}(\bx)}{\int_{D(0,2) \setminus D(0,2-\delta)} V_{\varepsilon}(\by) \, \dd \by },
\end{align*}
for every $\bx \in D(0,2-\delta) \setminus 0$ and $\delta>0$ small enough such that $\int_{D(0,2) \setminus D(0,2-\delta)} V_{\varepsilon}(\by) \, \dd \by>0$ and $\int_{D(0,2-\delta)} V_{\varepsilon}(\by) |\by|^{-2\beta} \, \dd \by>0.$ Now, by \eqref{eq:keyassumnegativebeta}, we can apply the Lebesgue's dominated convergence to imply that
\begin{align*}
 \lim_{n \to \infty} \int_{\R^2} \left |V_{\varepsilon} e^{\psi_{\varepsilon}} - V_{\varepsilon} e^{\psi_{n,\varepsilon}} \right| =0.
\end{align*}
In conclusion, \begin{align}
\label{eq:equationforpsinegative}
    - \Delta \psi_{\varepsilon} =4 \pi \beta V_{\varepsilon} e^{\psi_{\varepsilon}},
\end{align}
weakly in $\R^2$ and 
\begin{align}
\label{eq:totalcurvature}
    \int_{\R^2}V_{\varepsilon} e^{\psi_{\varepsilon}} =\lim_{n \to \infty} \int_{\R^2} V_{\varepsilon} e^{\psi_{n,\varepsilon}} = 1.
\end{align}
Moreover, by applying the same argument as above, we can let $\varepsilon \to 0$, to obtain a radially symmetric solution $\psi \in W^{1,q}_{\loc}(\R^2)$, for every $0<q<2$, to 
\begin{equation}
\label{eq:equationforpsinegativeendingsolution}
    \begin{aligned}
&-\Delta \psi = 4 \pi \beta\, V e^{\psi}, \quad \textup{weakly in } \R^2, \\
& \int_{\R^2} V e^{\psi} =1.
    \end{aligned}
    \end{equation}
Second, to prove that $\psi \in H^1_{\loc}(\R^2) \cap C^1(\R^2 \setminus 0)$, we first consider that, by \eqref{eq:upperbound} and \eqref{eq:equationforpsinegative}, $\psi$ is radially increasing and  
\begin{align*}
\psi(r) \leq \psi(2+r) \leq -2\beta \log (1+r) +C,    
\end{align*}
 for every $r>0$, where 
 \begin{align*}
     C:=-\log \left(\int_{D(0,2)} V(\by)\, |\by|^{-2\beta} \right). 
 \end{align*}
 Then, 
 \begin{align*}
     |\psi(r)| e^{\psi(r)} \leq 1+ e^{-2\beta \log (1+r) + C},
 \end{align*}
 for every $r>0.$ In conclusion, by \eqref{eq:gradphibound} and \eqref{eq:equationforpsinegative}, we have
 \begin{equation}
 \label{eq:uniformH1bound}
\begin{aligned}
 \int_{0}^R |\partial_r \psi|^2 \, r \dd r &=r (\partial_r \psi) \psi \bigg |_{r=R} -\int_0^R \frac{1}{r} \partial_r (r \partial_r \psi) \, \psi \, r \dd r \\&\leq -2\beta(-2\beta \log (1+R) + C)- 4\pi \beta 
 \int_0^R V e^{\psi} \, |\psi| \, r \dd r
 \\ & \leq -2\beta (-2\beta \log (1+R) + C)- 2 \beta \left(1+ e^{C}\right) \int_{D(0,R)} V,
\end{aligned}
\end{equation}
for every $R>0$. Moreover, by Proposition \ref{prop:logbound}, we get $\psi \in C^1(\R^2 \setminus 0)$. Hence, by taking the integral of the 
equation in \eqref{eq:equationforpsinegativeendingsolution}, we derive that 
\begin{align*}
    \lim_{\bx \to \infty} \frac{\psi(\bx)}{\log |\bx|} = -2\beta.
\end{align*}

Finally, we take $\alpha(V)<\infty$ and $\beta := -\alpha(V)$. Consider $-\alpha(V)<\beta_i<0$ be a strictly decreasing sequence converging to $\beta$, and $\psi_{i} \in H^1_{\loc}(\R^2) \cap C^1(\R^2 \setminus 0)$ be a sequence of functions satisfying 
\begin{equation}
\label{eq:approxnegbetadecreasing}
\begin{aligned}
   & -\Delta \psi_i = 4 \pi \beta_i  V e^{\psi_i},\quad \textup{weakly in } \R^2,\\
    & \int_{\R^2}  V e^{\psi_i}  = 1,\\
    & \lim_{\bx \to \infty} \frac{\psi_i(\bx)}{\log |\bx|} = -2\beta_i,
\end{aligned}
\end{equation}
for every $i \in \N.$ We show that
\begin{align}
\label{eq:increasingproperty}
    \psi_{i+1} + \log |\beta_{i+1}| > \psi_{i} + \log |\beta_i|, \quad \textup{a.e. } \R^2,
\end{align}
for every $i \in \N.$ To see this, fix $i \in \N$ and by construction of $\psi_i$, without loss of generality, we assume that $V$ is smooth. Then, by \eqref{eq:approxnegbetadecreasing} and $\beta_i$ being strictly decreasing sequence, we note that
\begin{align}
\label{eq:increasingpropertyousidelargedisk}
    \psi_{i+1} + \log |\beta_{i+1}| > \psi_{i} + \log |\beta_i|, \quad \textup{in } \R^2 \setminus D(0,R),
\end{align}
for large enough $R>0$. Moreover, $\rho = \psi_{i+1} + \log |\beta_{i+1}|- (\psi_{i} + \log |\beta_i|)$ satisfies
\begin{align*}
   -\Delta \rho + g \rho =0, \quad \textup{in } \R^2,
\end{align*}
where 
\begin{align*}
    g:= 4\pi V \frac{e^{\psi_{i+1} + \log |\beta_{i+1}|}-e^{\psi_{i} + \log |\beta_{i}|}}{\psi_{i+1} + \log |\beta_{i+1}|- (\psi_{i} + \log |\beta_i|)} \geq 0.
\end{align*}
Hence, by applying the maximum principle for $\rho$ in $D(0,R)$; see \cite[Ch. 6.4.1, Thm. 2; Ch. 6.4.2, Thm. 4]{E-10}, we get \eqref{eq:increasingproperty}. Hence,  by \eqref{eq:upperbound}, \eqref{eq:lowerbound}, and \eqref{eq:uniformH1bound}, $\psi_{i}+\log|\beta_{i}|$ is an increasing sequence and $\|\psi_{i}+\log|\beta_{i}|\,\|_{H^{1}(D)}$ is uniformly bounded for every bounded disk $D$. Then, by Rellich-Kondrakov theorem, $\psi_{i+1}+\log|\beta_{i}|$ converges pointwise and in $L^2_{\loc}(\R^2)$ to $\psi + \log|\beta| \in H^1_{\loc}(\R^2)$ and, by monotone convergence theorem and \eqref{eq:approxnegbetadecreasing}, satisfies 
\begin{equation*}
\begin{aligned}
   & -\Delta \psi = 4 \pi \beta  V e^{\psi},\quad \textup{weakly in } \R^2,\\
    & \int_{\R^2}  V e^{\psi}  = 1,\\
    & \lim_{\bx \to \infty} \frac{\psi(\bx)}{\log |\bx|} = -2\beta,
\end{aligned}
\end{equation*}
which completes the proof for $\beta<0.$

 Now, assume that $\beta >0$, $V \in L^1_{\loc}(\R^2)$, $V \geq C$ in $D(0,R_2) \setminus D(0,R_1)$ for some constants  $C>0, R_2> R_1>0$, and 
 \begin{equation}
 \begin{aligned}\label{eq:keyassumptionpositivebeta}
& \int_{\R^2 \setminus D(0,1)} V^-(\bx) \, |\bx|^{-2\beta} \, \dd \bx < \infty,\\  
&\int_{D(0,1)} V^+(\bx) |\bx|^{-\beta-\delta} \, \dd \bx + \int_{\R^2 \setminus D(0,1)} V^+(\bx) \, |\bx|^{-\beta + \delta} \, \dd \bx< \infty,\end{aligned}
\end{equation} for some $\delta >0.$ Let us consider a fixed radially symmetric function $\psi_0 \in C^{\infty}(\R^2)$ such that $\psi_0(\bx) = 2\beta \log |\bx|$ for every $\bx \in \R^2 \setminus D(0,1)$. Defining $$f:= -\Delta \psi_0 \in C^{\infty}_c(D(0,1)),$$ we have 
\begin{align}
\label{eq:integraloff}
    \int_{\R^2} f =- \int_{D(0,1)} \Delta \psi_0 = - 2 \beta  \int_{ 0}^{2\pi} r \partial_r \log r \big |_{r= 1} = -4 \pi \beta .
\end{align}
To obtain a radially symmetric function $\psi \in H^1_{\loc}(\R^2) \cap C^1(\R^2\setminus 0)$, satisfying
\begin{align*}
  &  \int_{\R^2} |V| e^{\psi}< \infty, \, \int_{\R^2} V e^{\psi}=1,\\
   & -\Delta \psi =4 \pi \beta V e^{\psi},\quad \textup{weakly in } \R^2,
\end{align*}
 by taking $\psi_1:= \psi + \psi_0$, it is enough to find a radially symmetric function $\psi_1 \in H^1_{\loc}(\R^2) \cap C^1(\R^2 \setminus 0)$ which satisfies 
\begin{align*}
   &\int_{\R^2} |V_1| e^{\psi_1}< \infty, \, \int_{\R^2} V_1 e^{\psi_1}=1,\\
    &-\Delta \psi_1 = 4 \pi \beta V_1 e^{\psi_1}  +f, \quad \textup{weakly in } \R^2,
\end{align*}
where $V_1 := V e^{-\psi_0}.$ This approach has been introduced in \cite{M-85}, and we apply it here to obtain a variational problem which is invariant under translation on the image.
Note that, by \eqref{eq:keyassumptionpositivebeta}, we have
\begin{equation}
\begin{aligned}\label{eq:keypropertyforV}
&\int_{\R^2} V_1^-(\bx) \, \dd \bx <\infty,\\
&\int_{\R^2} V_1^+(\bx)\, e^{(\beta+\delta) |\log|\bx|\,|}\, \dd \bx <\infty.\end{aligned}
\end{equation}
Let $n \in \N$ be large enough such that $D(0,R_2) \subset D(0,n)$. First, we find a radially symmetric function $\phi_n \in H^1_0(D(0,n)) \cap C^1(D(0,n) \setminus 0)$, satisfying \begin{equation} \begin{aligned}
\label{eq:crucialbounds}
    \|\nabla \phi_n\|_{L^2(D(0,n))} & \leq C,
\\  \int_{D(0,n)} V_1 e^{\phi_n-\phi_n(1)}& \geq \frac{1}{2C},\\
 \int_{D(0,n)} |V_1| e^{\phi_n-\phi_n(1)}& \leq 2C,
 \end{aligned}
\end{equation}for a constant $C>0$ independent of $n$, and
\begin{align}
\label{eq:equationforphin}
    -\Delta \phi_n =4 \pi \beta \frac{V_1 e^{\phi_n}}{ \int_{D(0,n)} V_1 e^{\phi_n}}+ f, \quad \textup{ weakly in } D(0,n) .
\end{align}
We define $\phi_n \in H^1_0(D(0,n))$ as a minimizer of the functional \begin{align*}
\mathcal{E}_n[\varphi] :=  \frac{1}{2} \int_{D(0,n)} |\nabla \varphi|^2 - 4 \pi \beta \log \left ( \int_{D(0,n)} V_1 e^{\varphi} \right ) - \int_{D(0,n)} f \varphi,
\end{align*}
over the set $\cM_n$ which contains the radially symmetric functions $\varphi \in H^1_0(D(0,n))$ satisfying $\int_{D(0,n)} V_1 e^{\varphi} >0.$ To prove that a minimizer $\phi_n \in  \cM_n$ exists and satisfies \eqref{eq:crucialbounds}, we first
note that, since $V_1 \geq C$ in $D(0,R_2) \setminus D(0,R_1)$ for  $C>0, R_2> R_1>0$, the set $ \cM_n$ is non-empty. Now, it is enough to show that $\mathcal{E}_n[\varphi]$ is coercive and bounded over $\varphi \in  \cM_n$. To show the boundedness of $\mathcal{E}_n,$ we note that $\mathcal{E}_n[\varphi] =\mathcal{E}_n[\varphi+C] $ for every constant $C$ and $\varphi \in  \cM_n$ by \eqref{eq:integraloff}. Now, choose a radially symmetric $\varphi^M \in C^{\infty}(\R^2)$ such that $\varphi^M = M$ on $D\left(0,\frac{R_1+R_2}{2}\right) \setminus D(0,R_1)$, where $M>0$ is to be determined later, $\varphi^M(r) = -M $ for every $r > \frac{R_1+R_2}{2}$, and $\varphi^M(r)=0$ for every $r<R_1$. Then, by \eqref{eq:keypropertyforV},
\begin{align*}
  \infty> \int_{D(0,n)} V_1 e^{\varphi^M} \geq  \int_{D(0,R_1)} V_1 + e^{M} \int_{D\left(0,\frac{R_1 +R_2}{2} \right) \setminus D(0,R_1)} V_1 -e^{-M} \int_{\R^2 \setminus D(0,R_2)} |V_1|.
\end{align*}
Now, by using \eqref{eq:keypropertyforV} again, we can take $M$ large enough such that $\int_{D(0,n)} V_1 e^{\varphi^M} \geq 1$ for every $n.$ Hence, $\varphi^M \in \cM_n$ and
\begin{equation}
\begin{aligned}
\label{eq:upperboundonenergry}
\cE_n[\varphi^M]  &\leq \int_{\R^2} |\nabla \varphi^M|^2 +\int_{D(0,R_2) } |f| \, |\varphi^M|
\\ &\leq \int_{D(0,R_2)} |\nabla \varphi^M|^2 + \|f\|_{L^{\infty}(D(0,R_2))} \int_{D(0,R_2)} |\varphi^M|   < C < \infty,
\end{aligned}
\end{equation}
where $C>0$ is a constant independent of $n$, which completes the proof that the functional $\cE_n$ over $ \cM_n$ is bounded from above. Now, define the following energy functional 
\begin{align*}
    \cE_n^+[\varphi] = \frac{1}{2} \int_{D(0,n)} |\nabla \varphi|^2 - 4 \pi \beta \log \left ( \int_{D(0,n)} V^+_1 e^{\varphi} \right ) - \int_{D(0,n)}f \varphi,
\end{align*}
for every $\varphi \in  \cM_n$. Then,
\begin{align}
\label{eq:firstenergybound}
    \mathcal{E}^+_n[\varphi] \geq \frac{\varepsilon}{2} \int_{D(0,n)} |\nabla \varphi|^2 + \theta_n,
\end{align}
for every $\varphi \in  \cM_n$, where 
\begin{align*}
     \theta_n := \frac{(1-\varepsilon)}{2} \int_{D(0,n)} |\nabla \varphi|^2 -  4 \pi \beta \log \left ( \int_{D(0,n)} V^+_1 e^{\varphi}\right) -\int_{D(0,n)} f \varphi.
\end{align*} 
Next, we aim to find a bound for $\theta_n$ from below, which is independent of $\varphi$. Note that by Proposition \ref{prop:logbound}, it is obtained that
\begin{align*}
    |\varphi(r) -\varphi(1)|^2 \leq \frac{\int_{D(0,R)} |\nabla \varphi|^2}{2\pi}  |\log r|,
\end{align*}
for every $0<r\leq R\leq n.$  Hence, by using \eqref{eq:integraloff}, for every $\alpha>0$ and $\varphi \in \mathcal{M}_n$ we have
\begin{align*}
    &\theta_n = \frac{(1-\varepsilon)}{2} \int_{D(0,n)} |\nabla \varphi|^2 -  4 \pi \beta \log \left ( \int_{D(0,n)} V^+_1 e^{\varphi - \varphi(1)}\right) -\int_{D(0,n)} f (\varphi-\varphi(1))  
    \\& \geq \frac{(1-\varepsilon)}{2} \int_{D(0,n)} |\nabla \varphi|^2- 4 \pi \beta \log \left ( \int_{D(0,n)} V^+_1 e^{\left( 2\pi (\beta + \delta)\frac{(\varphi(r) - \varphi(1))^2}{\int_{D(0,n)} |\nabla \varphi|^2} +  \frac{\int_{D(0,n)} |\nabla \varphi|^2}{8\pi (\beta+\delta)} \right)}\right)
    \\ & -\frac{\alpha}{2}\int_{D(0,1)} f^2 \int_{D(0,1)} |\nabla \varphi|^2 - \frac{1}{2\alpha} \int_{D(0,1)} \frac{(\varphi(r) - \varphi(1))^2}{\int_{D(0,1)} |\nabla \varphi|^2} \\&  \geq \frac{1}{2} \left(1-\varepsilon-\frac{\beta}{\beta+\delta} - \alpha \int_{D(0,1)} f^2  \right) \int_{D(0,n)} |\nabla \varphi|^2
    - 4\pi \beta \log \left(\int_{\R^2} V^+_1 e^{(\beta +\delta)|\log r|} \right) \\& - \frac{1}{4\alpha \pi} \int_{D(0,1)} |\log |\bx| \,| \, \dd \bx.
\end{align*}
Now, since $1-\varepsilon - \frac{\beta}{\beta+\delta} >0$, we can choose 
\begin{align*}
    0<\alpha  \leq 2\frac{1-\varepsilon - \frac{\beta}{\beta+\delta}}{\int_{D(0,R_2)} f^2},
\end{align*}
to get 
\begin{align}
\label{eq:thetabound}
    \theta_n \geq - 4\pi \beta \log \left(\int_{\R^2} V^+_1 e^{(\beta+\delta) |\log r|} \right) - \frac{1}{4\alpha \pi} \int_{D(0,1)} |\log |\bx|\,| \, \dd \bx.
\end{align}
In conclusion, by \eqref{eq:firstenergybound}, we derive
\begin{align}
\label{eq:finalengergybound}
   \mathcal{E}_n[\varphi] \geq  \mathcal{E}^+_n[\varphi] \geq \frac{\varepsilon}{2} \int_{D(0,n)} |\nabla \varphi|^2  - 4\pi \beta \log \left(\int_{\R^2} V^+_1 e^{(\beta+\delta) |\log r|} \right) - \frac{1}{4\alpha \pi} \int_{D(0,1)} |\log |\bx|\,|\, \dd \bx,
\end{align}
for every $\varphi \in  \cM_n$. This concludes that $\cE_n$ is coercive over $ \cM_n$. Hence,  there exists a minimizer $\phi_n \in \cM_n$ for the functional $\mathcal{E}_n$ over $ \cM_n$. Moreover, by \eqref{eq:upperboundonenergry} and \eqref{eq:finalengergybound}, we have
\begin{align}
\label{eq:energybounds}
 -C_1 \leq \mathcal{E}^+_n[\phi_n] \leq C_1,\quad -2C_1 \leq \mathcal{E}_n[\phi_n] \leq 2C_1,\quad \int_{D(0,n)} |\nabla \phi_n|^2 \leq 2C_1,
\end{align}
where $C_1$ is a constant independent of $n$. Now, by Proposition \ref{prop:logbound} and \eqref{eq:energybounds}, it is obtained that $\phi_n \in C^1(D(0,n) \setminus 0)$ and
\begin{align}
\label{eq:pointwisebound}
    |\phi_n(r) -\phi_n(1)|^2 \leq  \frac{C_1}{\pi} |\log r|, \quad \textup{for every } 0<r<n.
\end{align}
 Hence, 
\begin{align*}
    \left|\int_{D(0,R_2)} f (\phi_n - \phi_n(1)) \right| \leq \sqrt{\frac{C_1}{\pi}} \int_{D(0,R_2)} f(\bx) \sqrt{|\log|\bx|\,|}\, \dd \bx< \infty.
\end{align*}
Then, by \eqref{eq:integraloff} and \eqref{eq:energybounds}, we obtain \eqref{eq:crucialbounds} for large enough $C$ independent of $n$. In the next step, we derive that the functions $\widetilde{\phi}_n := \phi_n- \log \left(\int_{D(0,n)} V_1 e^{\phi_n}\right)$ converge pointwise and in $L^{2}_{\loc}(\R^2)$ to $\widetilde{\phi} \in \dot{H}(\R^2)$. To see this, we use \eqref{eq:crucialbounds} and \eqref{eq:pointwisebound} to obtain
\begin{align*}
    \|\widetilde{\phi}_n\|_{W^{1,2}(\Omega)} \leq C_2,
\end{align*}
for every fixed bounded open subset $\Omega \subset D(0,n)$, where $C_2$ is a positive constant depending on $\Omega,\delta,V,R_1,R_2$. By using Rellich-Kondrakov theorem, we derive that, up to a subsequence, $\Tilde{\phi}_n$ converges pointwise and strongly in $L^2_{\loc}(\R^2)$ to $\Tilde{\phi} \in W^{1,2}_{\loc}(\R^2)$. Moreover, using \eqref{eq:crucialbounds}, we obtain that $\widetilde{\phi} \in \dot{H}(\R^2)$ and
\begin{align*}
    \|\nabla \widetilde{\phi}\|^2_{L^2(\R^2)} \leq C.
\end{align*}
In the next step, we prove that $\int_{\R^2} |V_1| e^{\widetilde{\phi}}< \infty, \int_{\R^2} V_1 e^{\widetilde{\phi}} \geq 1$, and
\begin{align}
\label{eq:finalequation}
    - \Delta \widetilde{\phi} = 4 \pi \beta\,\frac{V_1 e^{\widetilde{\phi}}}{\int_{\R^2}  V_1 e^{\widetilde{\phi}} }+f,\quad \textup{weakly in } \R^2.
\end{align}
 First, we use \eqref{eq:keyassumptionpositivebeta}, \eqref{eq:crucialbounds}, and \eqref{eq:pointwisebound} to obtain that 
\begin{align*}
      & V^+_1(\bx) e^{\widetilde{\phi}_n(\bx)} =  V^+_1(\bx) e^{\widetilde{\phi}_n(\bx)}  1_{D(0,1)}(\bx) + V^+_1(\bx) 
      e^{\widetilde{\phi}_n(\bx)} 1_{\R^2 \setminus D(0,1)}(\bx) 
    \\ & =
    \frac{1}{\int_{D(0,n)} V e^{\phi_n- \phi_n(1)}}  \left(V^+_1(\bx) e^{\phi_n(\bx)- \phi_n(1)}  1_{D(0,1)}(\bx) + V^+_1(\bx) 
      e^{\phi_n(\bx)- \phi_n(1)} 1_{\R^2 \setminus D(0,1)}(\bx)\right)
    \\ & \leq 2C  \left(e^{\frac{C_1}{4 (\beta+\delta) \pi}}  V^+_1(\bx) \,|\bx|^{-\beta-\delta} 1_{D(0,1)}(\bx)+   V^+_1(\bx) \,e^{\sqrt{\frac{C_1}{\pi} \log |\bx|}} 1_{\R^2 \setminus D(0,1)}(\bx)\right),
\end{align*}
for every $\bx \in \R^2 \setminus 0$, and, by \eqref{eq:keypropertyforV}, we get
\begin{align*}
   e^{\frac{C_1}{4 (\beta+\delta) \pi}} \int_{D(0,1)} V_1^+(\bx) \,|\bx|^{-\beta-\delta} + \int_{\R^2 \setminus D(0,1)}  V^+_1(\bx)\,e^{\sqrt{\frac{C_1}{\pi} \log |\bx|}} \, \dd \bx< \infty.
\end{align*}
Then, by Lebesgue's dominated convergence
\begin{align*}
    \lim_{n \to \infty} \int_{\R^2} V^+_1 e^{\widetilde{\phi}_n} = \int_{\R^2} V^+_1 e^{\widetilde{\phi}}.
\end{align*}
Hence, by \eqref{eq:crucialbounds} and Fatou's lemma, we obtain 
\begin{align*}
    \int_{\R^2} |V_1| \, e^{\widetilde{\phi}} < \infty, \, \int_{\R^2} V_1\, e^{\widetilde{\phi}} \geq \limsup_{n \to \infty} \int_{\R^2} V_1\, e^{\widetilde{\phi}_n} =1.
\end{align*}
Define
\begin{align*}
    \cE[\varphi] := \int_{\R^2} |\nabla  \varphi|^2 - 4\pi \beta \log \left(\int_{\R^2} V_1 e^{\varphi} \right) -\int_{\R^2} f \varphi, 
\end{align*}
for every $\varphi \in \dot{H}^1(\R^2)$ which satisfies $\int_{\R^2} |V_1| e^{\varphi}< \infty$ and $\int_{\R^2} V_1 e^{\varphi}>0$. Then,
\begin{align}
\label{eq:minimizerofenergy}
  \mathcal{E}[\widetilde{\phi}]  \leq \liminf_{n \to \infty} \mathcal{E}_n[\widetilde{\phi}_n] \leq \liminf_{n \to \infty} \mathcal{E}_n[\varphi]  \leq \mathcal{E}[\varphi],
\end{align} 
for every radially symmetric $\varphi \in C^{\infty}_c(\R^2)$ such that $\int_{D(0,n)} V_1 e^{\varphi}>0$ for $n$ large enough. Now, we can approximate $\widetilde{\phi} + t \varphi$ for a radially symmetric $\varphi \in C^{\infty}_c(\R^2)$ and small enough $t \in \R$ with radially symmetric functions $\widetilde{\varphi}_n \in C^{\infty}_c(\R^2)$ such that $\int_{D(0,n)} V_1 e^{\widetilde{\varphi}_n}>0$ for $n$ large enough and
\begin{align*} 
     \mathcal{E}[\widetilde{\phi}+t \varphi] = \lim_{n \to \infty}\mathcal{E}[ \widetilde{\varphi}_n ].
\end{align*}
Note that we used 
\begin{align*}
    \int_{\R^2} |V_1|\, \left|e^{\widetilde{\phi} + t \varphi} - e^{\widetilde{\phi}}\right| \leq \|e^{t \varphi} -1\|_{L^{\infty}(\R^2)} \int_{\R^2} |V_1|\, e^{\widetilde{\phi}},
\end{align*}
which goes to zero as $t \to 0$. Hence, by \eqref{eq:minimizerofenergy}, we conclude that 
\begin{align*}
     \mathcal{E}[\widetilde{\phi}] \leq  \mathcal{E}[\widetilde{\phi}+t \varphi],
\end{align*}
for every radially symmetric $\varphi \in C^{\infty}_c(\R^2)$ and small enough $t$, depending on $\varphi.$ In particular, $\widetilde{\phi}$ is a critical point for $\mathcal{E}$ and satisfies \eqref{eq:finalequation}. Finally, by taking $\psi_1 := \widetilde{\phi} - \log \left(\int_{\R^2} V_1 e^{\widetilde{\phi}} \right)$ and using Proposition \ref{prop:logbound}, we complete the proof.
\end{proof}

\begin{remark}
    In the case of $\beta>0$, we proved in fact a stronger property of solutions as follows: Define $\psi_0,V, \psi$ as in the proof of Theorem \ref{thm:maintheorem}, $V_1 =V e^{-\psi_0}$, and $f := -\Delta \psi_0$. Then, $\psi +\psi_0$ is a minimizer of the energy functional 
\begin{align*}
    \cE[\varphi] := \frac{1}{2} \int_{\R^2} |\nabla \varphi|^2 - 4 \pi \beta \log \left ( \int_{\R^2} V_1 e^{\varphi} \right ) - \int_{\R^2} f \varphi,
\end{align*}
among all the radially symmetric $\varphi \in \dot{H}(\R^2)$ such that $\int_{\R^2} |V_1| e^{\varphi}< \infty$, $\int_{\R^2} V_1 e^{\varphi} >0.$ Moreover, by \eqref{eq:finalengergybound}, for every  $\varphi \in \dot{H}(\R^2)$, satisfying $\int_{\R^2} |V_1| e^{\varphi}< \infty$, $\int_{\R^2} V_1 e^{\varphi} >0$, we have 
\begin{align*}
   \cE[\varphi]  \geq  \cE^+[\varphi] \geq \frac{\delta}{4(\beta+\delta)} \int_{\R^2} |\nabla \varphi|^2 + C,
\end{align*}
    where 
\begin{align*}
     \cE^+[\varphi]:=  \frac{1}{2} \int_{\R^2} |\nabla \varphi|^2 - 4 \pi \beta \log \left ( \int_{\R^2} V_1^+ e^{\varphi} \right ) - \int_{\R^2} f \varphi,
\end{align*}
and 
    \begin{align*}
        C:= -4\pi \beta \log \left( \int_{\R^2} V_1^+ e^{(\beta+\delta) |\log r|} \right) +\frac{\beta (\beta+\delta)}{\delta} \int_{D(0,1)} \log |\bx| \, \dd \bx.
    \end{align*}
\end{remark}

\begin{remark}
\label{rem:extracondiboundnegativebeta}
    Let $\beta<0$ and $V$ in Theorem \ref{thm:maintheorem} satisfy the extra condition 
\begin{align}
\label{eq:extracondiH1dot}
    \int_{\R^2 \setminus D(0,1)} V(\bx) \, |\bx|^{-2\beta}  \, \log|\bx| \, \dd \bx < \infty.
\end{align}
    Then,
\begin{align*}
    \int_{0}^{\infty} |\partial_r (\psi + \psi_0)|^2 \, r \, \dd r< \infty.
\end{align*}
To prove this, we use the integral of the equation \eqref{eq:mainequation} and radial symmetric property of $\psi$ to derive that 
\begin{equation}
\label{eq:boundandassympsi}
\begin{aligned}
&\lim_{r \to \infty} \frac{\psi(r)}{\log r} = -2\beta,\\
&\psi(r) \geq \psi(1), \quad \textup{for every } r \geq 1,\\
 &   \psi(r) \leq \psi(1) - 2\beta |\log r|, \quad \textup{for every } r >0,
\end{aligned}
\end{equation}
and, for every $r>1$, we have
\begin{align*}
    &\left|-r \partial_r \psi(r) - 2\beta\right| = -2 \beta \left |- \int_{D(0,r)} V e^{\psi} +  \int_{\R^2} V e^{\psi}\right|
     \\ &\leq  \frac{-2 \beta }{\log r} \int_{\R^2 \setminus D(0,r)} V(\bx) e^{\psi(\bx)} \log |\bx| \, \dd \bx
     \\ & \leq \frac{-2 \beta e^{\psi(1)} }{\log r} \int_{\R^2} V(\bx)\, |\bx|^{-2\beta}  \log |\bx| \, \dd \bx< \infty.
\end{align*}
Hence, by \eqref{eq:mainequation}, \eqref{eq:extracondiH1dot}, and \eqref{eq:boundandassympsi}, it is obtained that
\begin{align*}
&\int_{1}^{\infty}\left|\partial_r \left( \psi(r) + \psi_0(r) \right) \right|^2 \, r \, \dd r+ \psi(1)  (\partial_r( \psi)(1) +2\beta) \\
 &\leq\limsup_{R \to \infty}\int_{1}^{R}\left|\partial_r \left( \psi(r) + \psi_0(r) \right) \right|^2 \, r\, \dd r -r \left( \psi(r) + \psi_0(r) \right) \partial_r \left( \psi(r) + \psi_0(r) \right) \bigg |^R_{1}\\
 \\&\leq \limsup_{R \to \infty} \int_{1}^{R} -\frac{1}{r} \partial_r (r \partial_r \left( \psi(r) + \psi_0(r) \right))  \left( \psi(r) + \psi_0(r) \right) \, r \dd r \\&\leq  \limsup_{R \to \infty} 4 \pi \beta \int_{1}^{R} V(r) e^{\psi(r)} \left( \psi(r) + \psi_0(r) \right)\, r \, \dd r  \\&\leq - 4 \pi \beta  \int_{1}^{\infty}  V(r) e^{\psi(1)-2\beta \log r} (|\psi(1)| - 2\beta \log(r)) \, r \, \dd r < \infty,
 \end{align*}
which completes the proof.

\end{remark}
\begin{remark}
To obtain a uniform bound on $|\psi+\psi_0- \psi(1)|$ in Theorem \ref{thm:maintheorem}, it is sufficient to assume that 
\begin{align}
\label{eq:extracondpointwiseboud}
     \int_{\R^2 \setminus D(0,1)} |V(\bx)| \, |\bx|^{-2\beta}  \, |\,\log|\bx|\, |^{1+\varepsilon} \, \dd \bx < \infty,
\end{align}
for some $\varepsilon>0,$ where here $\beta$ can be both positive or negative. Since the idea of the proof is the same as the previous remark, we skip it here. We note that also by proof of Theorem \ref{thm:comparison}, 
\begin{align*}
    |\psi(1)| \leq \limsup_{n \to \infty}  \left|\phi_n(1)- \log \left(\int_{D(0,n)} V_1 e^{\phi_n}\right)\right| \leq \log (2C),
\end{align*}
    where $C$ is a constant depending continuously on $R_1,R_2,\delta,\beta, \int_{\R^2} V_1^+ e^{(\beta+\delta) \log r}, \int_{\R^2} V_1^-$, where $V_1 =V e^{-\psi_0}$ and $\psi_0$ is as in Theorem \ref{thm:maintheorem}.
    
\end{remark}

Now, we provide a few examples to discuss the necessity of the conditions in Theorem \ref{thm:maintheorem}.
\begin{example}
\label{ex:blowupenergy}
   We provide an example to show that $V$ becoming zero fast enough at the origin plays a key role in the proof of Theorem \ref{thm:maintheorem}. Let $\varepsilon,C,\beta$ be positive constants and $V(r)= r^{\alpha} V_0(r)$ for every $r \geq 0$, where $V_0 \in C(\R)$ is non-negative, $V_0>C$ in $D(0,\varepsilon)$, and $0 \leq \alpha < \beta -2.$ Then,
   \begin{align*}
       \inf \cE[\varphi] = \inf  \frac{1}{2} \int_{D(0,1)} |\nabla \varphi|^2 - 4 \pi \beta \log \left ( \int_{D(0,1)} V e^{\varphi} \right ) = -\infty,
   \end{align*}
where we take the infimum over the radially symmetric functions $\varphi \in H^1_0(D(0,1))$. To see this, for every $0<\delta<\varepsilon$, define 
\begin{align*}
    \varphi_{\delta}(r) := \begin{cases}
-2\beta \log r, \quad \textup{for every } \delta \leq r \leq 1,\\ 
-2\beta \log \delta + 4\pi \beta \log (\alpha+2), \quad \textup{for every } 0\leq r< \delta.
    \end{cases}
\end{align*}
It is easy to check that $\varphi_{\delta} \in H^1_0(D(0,1))$. Then, by $\alpha < \beta-2$ and $V_0>C$ in $D(0,\varepsilon)$, we have
\begin{align*}
     \cE[\varphi_{\delta}] &= -4\pi \beta \log(2\pi ) -4\pi \beta  \log \left( \delta^{-\beta}\int_0^{\delta} r^{\alpha+1} V_0(r)\, \dd r + \delta^{\beta} \int_{\delta}^1  r^{\alpha-2\beta+1} V_0(r) \, \dd r\right)
     \\ &\leq -4\pi \beta \log(2\pi C) +4\pi \beta (\beta-2-\alpha)  \log  \delta.
\end{align*}
Now, by letting $\delta \to 0$ and using $\alpha < \beta-2$, we obtain
\begin{align*}
    \lim_{\delta \to 0}  \cE[\varphi_{\delta}] = -\infty.
\end{align*}
   
\end{example}

\begin{example}
\label{ex:nonradialminimizing}
   Let $\beta>2$ and $V \in C(D(0,1))$ be non-negative, satisfying $V(x_0) >C$ for some $C>0,\bx_0 \in D(0,1)$. We show that 
    \begin{align}
    \label{eq:-inifinitynonradialminimizer}
      \inf \cE[\varphi] = \inf  \frac{1}{2} \int_{D(0,1)} |\nabla \varphi|^2 - 4 \pi \beta \log \left ( \int_{D(0,1)} V e^{\varphi} \right ) = -\infty,
   \end{align}
where the infimum is taken over $\varphi \in C^{\infty}_c(D(0,1))$. Without loss of generality, consider $V \geq C$ in $D(\bx_0,r) \subset D(0,1)$ for $r,C>0, \bx_0 \in D(0,1).$ Then,
\begin{align*}
      \cE[\varphi] \leq \frac{1}{2} \int_{D(\bx_0,r)} |\nabla \varphi|^2 - 4 \pi \beta \log \left ( \int_{D(\bx_0,r)}  e^{\varphi} \right ) - 4 \pi \beta \log C,
\end{align*}
for every $\varphi \in C^{\infty}_c(D(\bx_0,r)).$ Hence, by defining $\widetilde{\varphi}(\bx) := \varphi(\bx-\bx_0)$ for every $\bx \in \R^2$, we have
\begin{align*}
      \cE[\varphi] \leq \frac{1}{2} \int_{D(0,r)} |\nabla \widetilde{\varphi}|^2 - 4 \pi \beta \log \left ( \int_{D(0,r)}  e^{\widetilde{\varphi}} \right ) - 4 \pi \beta \log C.
\end{align*}
Finally, by taking the infimum over radially symmetric $\widetilde{\varphi} \in C^{\infty}_c(D(0,r))$ and using Example \eqref{ex:blowupenergy}, we get \ref{eq:-inifinitynonradialminimizer}.

\end{example}

\begin{example}
\label{ex:regularityexample}
    The following proves that the regularity in Theorem \ref{thm:maintheorem} is sharp. Define 
\begin{align}
    \psi(r) = \begin{cases}
-\frac{1}{2} \log(-\log r), \quad &\textup{for } r \leq \alpha,\\
-\frac{1}{2 \log \alpha } \log r + \frac{1}{2}-\frac{1}{2} \log(-\log \alpha), \quad &\textup{for } r > \alpha,
    \end{cases}
\end{align}
for $0<\alpha<1.$
    Then, $\psi \in H^1_{\loc}(\R^2) \cap C^{1}(\R^2 \setminus 0)$ and satisfies 
    \begin{equation}
    \begin{aligned}
    \label{eq:exeqaution}
        &-\Delta \psi = 4\pi \beta V e^{\psi}, \quad \textup{weakly in } \R^2, \\
        &\int_{\R^2} V e^\psi =1,
    \end{aligned}
    \end{equation}
    where $\beta = \frac{1}{4\log \alpha}$ and 
    \begin{align*}
        V(r) = -\frac{1}{8\pi \beta} \frac{1_{[0,\alpha]}}{r^2 (-\log r)^{\frac{3}{2}}}, \quad \textup{for every } r>0.
    \end{align*}
Note that $V \in L^1(\R^2)$ is non-negative and, by changing $\alpha$, $\beta$ can take any negative values. This gives an example where $\psi$ blows up at the origin and is continuously differentiable in $\R^2 \setminus 0.$

For the case of $\beta>0$, let $M>0$ to be chosen later. We choose a radially symmetric function $\psi \in H^1_{\loc}(\R^2) \cap C^{1}(\R^2 \setminus 0)$ such that
\begin{align}
-\Delta \psi(r) = \begin{cases}
-\frac{1}{2 r^2 (\log r)^2}, \quad &r \leq \frac{1}{2},\\
-\frac{2}{ (\log 2)^2} + M \left(r-\frac{1}{2}\right) (1 - r), \quad &\frac{1}{2}\leq r \leq 1,\\
0, \quad &r \geq 1.    
\end{cases}
\end{align}
Then, by taking $M$ large enough, we have $\psi$ satisfies \eqref{eq:exeqaution}, where $\beta := \frac{1}{4\pi} \int_{\R^2} -\Delta \psi >0$, $V:=- \frac{1}{4\pi \beta} e^{-\psi} \Delta \psi  \in 
L^1(\R^2)$ is positive somewhere in $\R^2$, and 
\begin{align*}
   & \int_{\R^2} V e^{\psi}=1,\\
   & \int_{\R^2} |V|\, e^{\psi} < \infty.
\end{align*}
 Moreover, since $-\Delta \psi =0$ for $r\geq 1,$ we derive that $r\partial_r\psi = C$ for $r \geq 1,$ where $C$ is a constant, satisfying
\begin{align*}
    C= \partial_r\psi(1) =\frac{1}{2\pi} \int_{\R^2} \Delta \psi = -2 \beta.
\end{align*}
Hence, by choosing $\psi_0 \in C^{\infty}(\R^2)$ such that $\psi(r) = 2\beta \log r$ for $r \geq 1$, we obtain that 
\begin{align*}
    \int_{\R^2} |\nabla (\psi + \psi_0)|^2 < \infty.
\end{align*}

\end{example}

Now, we prove Corollary \ref{cor:existencepositiveassymdecay}.

\begin{proof}[Proof of Corollary \ref{cor:existencepositiveassymdecay}]
If $n > \beta -2$, then $r^n V(r)$ satisfies the conditions in Theorem \ref{thm:maintheorem} and there exists a radially symmetric $\psi \in C^{0,\alpha}_{\loc}(\R^2)$ for some $\alpha>0$, by Proposition \ref{prop:regularity}, which satisfies \eqref{eq:radialLiouvillesol}. Moreover, 
by taking the integration of \eqref{eq:radialLiouvillesol}, we have 
\begin{align*}
   \lim_{r \to \infty} -r \partial_r \psi(r) = -2\beta. 
\end{align*}
Hence,
\begin{align*}
    \lim_{\bx \to \infty} \frac{\psi(\bx)}{\log|\bx|} = -2\beta.
\end{align*}
To prove the other direction, we use the Pokhozhaev identity and an argument in \cite{CL-93}. It is enough to consider that $\beta>0$ and $\psi \in L^1_{\loc}(\R^2)$ is a function such that \eqref{eq:equationexistencepositvedecay} holds. Then, by \eqref{eq:uniformbound} for $\varepsilon>0$  small enough and $p = 1+\varepsilon$, we have
\begin{align*}
    \int_{\R^2} |\bx|^{np} V^p(\bx) e^{p\psi(\bx)}\, \dd \bx \leq C  \int_{1}^{\infty} |r|^{np+1} V^p(r) r^{p (-2\beta + \varepsilon) } \, \dd r < \infty,
 \end{align*}
 for some $C>0$. Hence, by \cite[Lem. 3.6, Lem. 3.7, Lem. 3.15]{ALN-24} and the Sobolev embedding theorem, we have
 \begin{align*}
   \phi(\bx) :=-2\beta \int_{\R^2}  (\log|\bx-\by| - \log (|\by|+1) ) \, |\by|^{n} V(\by) e^{\psi(\by)} \, \dd \by,
\end{align*}
is a well-defined function, which satisfies
\begin{equation}
\label{eq:keypropphi}
\begin{aligned}
\phi &\in  W^{2,p}_{\loc}(\R^2),\\
-\Delta  \phi &= V e^{\psi}, \quad \textup{weakly in } \R^2,\\
  \lim_{\bx \to \infty}  \frac{\phi(\bx)}{\log|\bx|} &= -2\beta.
\end{aligned}
\end{equation}
In conclusion, $\psi - \phi$ is harmonic and, by \eqref{eq:equationexistencepositvedecay}, is bounded from above by $5\beta \log |\bx|+ C$ for a constant $C$ and $|\bx| \geq 1$. Thus, $\psi - \phi$ must be a constant. In the next step, we prove that 
\begin{align}
\label{eq:assymptoticradialtheta}
    \lim_{r \to \infty} r \partial_r \psi = -2\beta,\, \lim_{r \to \infty}  \partial_{\theta} \psi = 0,
\end{align}
where $(r,\theta)$ is the polar coordinate on $\R^2.$ As in \cite[Lem. 1.3]{CL-93}, note that, since $\psi - \phi$ is a constant, we obtain
\begin{align*}
   & \bx \cdot \nabla \psi(\bx) +2\beta =  -2\beta  \int_{\R^2}  \frac{\by \cdot (\bx-\by)}{|\bx-\by|^2} \, |\by|^{n} V(\by) e^{\psi(\by)} \, \dd \by,\\
    &\partial_{\theta} \psi(\bx) = 2\beta \int_{\R^2}  \frac{\by^{\perp} \cdot (\bx-\by)}{|\bx-\by|^2} \, |\by|^{n} V(\by) e^{\psi(\by)} \, \dd \by,
\end{align*}
for every $\bx \in \R^2,$ where $\by^{\perp} =(-y_2,y_1)$ for every $\by = (y_1,y_2) \in \R^2.$ Hence, to prove \eqref{eq:assymptoticradialtheta}, it is sufficient to demonstrate 
\begin{align}
\label{eq:keylimit}
  \lim_{\bx \to \infty}  \int_{\R^2}  \frac{|\by|^{n+1} V(\by)}{|\bx-\by|} \,  e^{\psi(\by)} \, \dd \by =0.
\end{align}
To show this, we write
\begin{align*}
   & \int_{\R^2}  \frac{|\by|^{n+1} V(\by)}{|\bx-\by|} \,  e^{\psi(\by)} \, \dd \by\\& = \int_{|\bx-\by| \geq \frac{|\bx|}{2}}  \frac{|\by|^{n+1} V(\by)}{|\bx-\by|} \,  e^{\psi(\by)} \, \dd \by
   + \int_{|\bx-\by| <\frac{|\bx|}{2}}  \frac{|\by|^{n+1} V(\by)}{|\bx-\by|} \,  e^{\psi(\by)} \, \dd \by = I_1 +I_2,
\end{align*}
for every $\bx \in \R^2$ and prove that each term $I_1,I_2$ converges to zero as $\bx$ goes to infinity. To estimate $I_1$, for a fixed $M>0$, we have
\begin{align*}
 \limsup_{\bx \to \infty}  I_1 &= \limsup_{\bx \to \infty}  \int_{|\bx-\by| \geq \frac{|\bx|}{2}, |\by| \leq M}  \frac{|\by|^{n+1} V(\by)}{|\bx-\by|} \,  e^{\psi(\by)} \, \dd \by +\int_{|\bx-\by| \geq \frac{|\bx|}{2}, |\by| > M}  \frac{|\by|^{n+1} V(\by)}{|\bx-\by|} \,  e^{\psi(\by)} \, \dd \by
 \\ & \leq  3 \int_{ |\by| > M}  |\by|^{n} V(\by) \,  e^{\psi(\by)} \, \dd \by.
\end{align*}
Since $M$ is arbitrary and $\int_{ \R^2}  |\by|^{n} V(\by) \,  e^{\psi(\by)} \, \dd \by=1,$ we obtain $\lim_{\bx \to \infty}  I_1=0.$ For the term $I_2$, we note that $|\bx -\by| < \frac{|\bx|}{2}$, for $\bx,\by \in \R^2$, implies 
\begin{align*}
         \frac{|\bx|}{2} <|\by| < \frac{3 |\bx|}{2}.
\end{align*}
Moreover, by $\psi- \phi$ being a constant and \eqref{eq:keypropphi}, we get 
\begin{align}
\label{eq:epsiboundabove}
    e^{\psi(\by)} \leq C_1 (1+|\by|)^{-2\beta +\delta},
\end{align}
 for every $\delta>0,\by \in \R^2$ and a constant $C_1>0$ independent of $\by$. Hence, since $V$ is a decreasing function, we conclude
\begin{align*}
    I_2 \leq C_2 V\left(\frac{\bx}{2} \right)  |\bx|^{n+2-2\beta+\delta},
\end{align*}
for every $\bx \in \R^2 \setminus D(0,1)$. Finally, by \eqref{eq:uniformbound2} and $V$ being decreasing, we have 
\begin{align}
\label{eq:assyplimitforV}
  \limsup_{r \to \infty}  V\left(\frac{r}{2} \right)  |r|^{n+2-2\beta+\delta} \leq  \limsup_{r \to \infty} \int_{D(0,r/2) \setminus D(0,r/4)}  t^{n+1-2\beta+\delta} V(t) \, \dd t=0,
\end{align}
which completes the proof of \eqref{eq:keylimit}. Now, by multiplying the equation \eqref{eq:equationexistencepositvedecay} with $\bx \cdot \nabla \psi$ and taking the integration on $D(0,M)$, for every $M>0$, we derive
\begin{align*}
  \frac{1}{4\pi \beta}  \int_{\partial D(0,M)}  r \left( \frac{|\nabla \psi|^2}{2} -|\partial_r \psi|^2 \right)
     &= -\int_{D(0,M)} e^{\psi(\bx)} \bx \cdot \nabla \left(|\bx|^n V(\bx) \right) \, \dd \bx\\& - 2 + \int_{\partial D(0,M)} |\bx|^n V(\bx) e^{\psi(\bx)} \, \dd \bx.
\end{align*}
Note that by \eqref{eq:epsiboundabove} and \eqref{eq:assyplimitforV}, we obtain 
\begin{align*}
    \lim_{M \to \infty}  \int_{\partial D(0,M)} |\bx|^n V(\bx) e^{\psi(\bx)} \, \dd \bx =0,
\end{align*}
and, by \eqref{eq:assymptoticradialtheta}, we imply that 
\begin{align*}
 \lim_{M \to \infty} \frac{1}{4\pi \beta}  \int_{\partial D(0,M)}  r \left( \frac{|\nabla \psi|^2}{2} -|\partial_r \psi|^2 \right) = - \beta.
\end{align*}
Hence, by $V$ being radially decreasing, we arrive at
\begin{align*}
\beta -2 &= \int_{\R^2}e^{\psi(\bx)} \bx \cdot \nabla \left(|\bx|^n V(\bx) \right)  \, \dd \bx 
\\& =n  + \int_{\R^2}(\partial_r V(r)) |\bx|^{n+1}   e^{\psi(\bx)} \, \dd \bx
< n,
\end{align*}
which completes the proof. 
\end{proof}

Now, we prove Theorem \ref{thm:strongestnegativebeta}

\begin{proof}[Proof of Theorem \ref{thm:strongestnegativebeta}]
Assume that $-\beta < \alpha(V)$. Define the regularization
\begin{align*}
    V_{\varepsilon} = \varphi_{\varepsilon} \ast \left(V 1_{D(0,1/\varepsilon)} \right) \in C^{\infty}_c(\R^2),
\end{align*}
for every $0<\varepsilon<1$ small enough such that $D \subset D(0,1/\varepsilon)$, where $\varphi_{\varepsilon}(\bx) =: \frac{\varphi(\bx/\varepsilon)}{\varepsilon^2}$ for every $\bx \in \R^2$ and $\varphi \in C^{\infty}_c(\R^2)$ is a non-negative radially symmetric function which satisfies 
\begin{align*}
    \int_{\R^2} \varphi =1.
\end{align*}
Then, $V_{\varepsilon} $ converges to $V$ pointwise and in $L^1_{\loc}(\R^2)$, up to a subsequence, and
\begin{equation}
\label{eq:conditiononapproxforV}
\begin{aligned}
  \lim_{\varepsilon \to 0}  \int_{\R^2} |V(\bx)-V_{\varepsilon}(\bx)|\, |\bx|^{-2\beta}\, \dd \bx =0,\\
   V_{\varepsilon} \geq C_1, \quad \textup{in } D(\bx_0,r_0),
\end{aligned}
\end{equation}
for some $\bx_0 \in \R^2, r>0.$ Consider $\widetilde{W}_1 \in C_{c}^{\infty}(D(0,r_0/2)) \setminus 0$ as a non-negative radially symmetric function, such that
\begin{align*}
    \widetilde{W}_1(\bx) \leq C_1 \leq V_{\varepsilon}(\bx_0 + \bx),\quad \textup{for a.e. } \bx \in \R^2, \, 0<\varepsilon<1,
\end{align*}
  Then, by Theorem \ref{thm:maintheorem}, there exists radially symmetric function $\widetilde{\psi}_{1} \in C^{\infty}(\R^2)$ such that 
\begin{align*}
 &-\Delta \widetilde{\psi}_{1} =-4 \pi  \widetilde{W}_{1} e^{\widetilde{\psi}_{1}}, \quad \textup{weakly in } \R^2,\\
 &\int_{\R^2}\widetilde{W}_{1} e^{\widetilde{\psi}_{1}}=1.
    \end{align*}
Define $\psi_{1}(\bx) := \widetilde{\psi}_{1}(\bx-\bx_0)-\log |\beta|, W_{1}(\bx) := \widetilde{W}_{1}(\bx-\bx_0)$ for every $\bx \in \R^2.$ Then,
\begin{equation}
\label{eq:approxsupersol}
\begin{aligned}
 &-\Delta \psi_{1} =4 \pi \beta W_{1} e^{\psi_{1}}, \quad \textup{weakly in } \R^2,\\
 &\int_{\R^2}W_{1} e^{\psi_{1}}=1.
    \end{aligned}
\end{equation}
Hence, by Lemma \ref{lem:degreeformula}, we derive
\begin{align}
\label{eq:lowerboundpsi1}
  -2\beta \log( |\bx|+1)+ \log |\beta| -C_2 \leq  \psi_{1}(\bx) \leq -2\beta \log( |\bx|+1)+\log |\beta| +C_2,
\end{align}
for every $\bx \in \R^2$, where $C_2>0$ is a constant depending on $D,C_1$. Now, for every $\bx \in \R^2$, define
\begin{align*}
    \psi_{2,\varepsilon}(\bx) := -2\beta \frac{\int_{\R^2} (\log |\bx - \by| - \log (|\by|+1)) V_{\varepsilon}(\by) (|\by|+1)^{-2\beta}  \, \dd \by}{\int_{\R^2} V_{\varepsilon}(\by) (|\by|+1)^{-2\beta} \, \dd \by }  - C_3,
\end{align*}
where 
\begin{align*}
    C_{3} := \sup_{0<\varepsilon<1}\max \left(-\log |\beta|-C_2,\log\left(\int_{\R^2} V_{\varepsilon}(\by) (|\by|+1)^{-2\beta} \, \dd \by\right)\right).
\end{align*}
Note that, by \cite[Lem. 3.7]{ALN-24}, we have $\psi_{2,\varepsilon}$ belongs to $C^{\infty}(\R^2)$ and satisfies
\begin{align*}
    -\Delta \psi_{2,\varepsilon}(\bx) = 4\pi \beta \frac{ V_{\varepsilon}(\bx) (|\bx|+1)^{-2\beta} }{\int_{\R^2} V_{\varepsilon}(\by) (|\by|+1)^{-2\beta} \, \dd \by }, \quad \textup{for every } \bx \in \R^2.
\end{align*}
Moreover, by using the inequality $|\bx-\by| \leq (|\bx|+1) (|\by| +1)$ for every $\bx,\by \in \R^2$, we have 
\begin{align}
\label{eq:upperboundpsi2}
   \psi_{2,\varepsilon}(\bx) \leq -2\beta \log (|\bx|+1)- C_3, \quad \textup{for every } \bx \in \R^2.
\end{align}
Hence, by \eqref{eq:lowerboundpsi1}, we obtain
\begin{align*}
&\psi_1 \geq \psi_{2,\varepsilon},\quad \textup{in } \R^2,\\
    \Delta \psi_{2,\varepsilon} +4\pi \beta V_{\varepsilon} e^{\psi_{2,\varepsilon}} &\geq 0 \geq \Delta \psi_{1} + 4\pi \beta V_{\varepsilon} e^{\psi_{1}}, \quad \textup{in } \R^2.
\end{align*}
In conclusion, by \cite[Thm. 2.10]{N-82}, there exists a function $\psi_{\varepsilon,\beta} \in C^{\infty}(\R^2)$
such that 
\begin{equation}
\label{eq:approxequanegabetageneral}
\begin{aligned}
\psi_1 &\geq \psi_{\varepsilon,\beta} \geq \psi_{2,\varepsilon}, \quad \textup{in } \R^2,\\
    -\Delta \psi_{\varepsilon,\beta} &=4 \pi \beta \,V_{\varepsilon} e^{\psi_{\varepsilon,\beta}}, \quad \textup{in } \R^2.
\end{aligned}
\end{equation}
Note that, by Lemma \ref{lem:degreeformula}, we have 
\begin{equation}
\label{eq:assymapproxsolepsilon}
\begin{aligned}
    &\int_{\R^2} V_{\varepsilon} e^{\psi_{\varepsilon,\beta}} = -\frac{1}{2 \beta}  \lim_{\bx \to \infty} \frac{\psi_{\varepsilon,\beta}(\bx)}{\log |\bx|} 
    \\ =& -\frac{1}{2 \beta}  \lim_{\bx \to \infty} \frac{\psi_{2,\varepsilon}(\bx)}{\log |\bx|} = -\frac{1}{2 \beta}  \lim_{\bx \to \infty} \frac{\psi_{1}(\bx)}{\log |\bx|} = 1. 
\end{aligned}
\end{equation}

Then, by the same argument as in \cite[Lem. 3.6]{ALN-24}, we have $\|\psi_{2,\varepsilon}\|_{L^1(\Omega)}$ is uniformly bounded for bounded subsets $\Omega \subset \R^2.$ Now, by using \eqref{eq:approxequanegabetageneral}, we have $\|\psi_{\varepsilon,\beta}\|_{L^1(\Omega)}$ is uniformly bounded for every bounded subset $\Omega \subset \R^2.$ Then, by \eqref{eq:approxequanegabetageneral} and a standard $L^1$-elliptic regularity result; see \cite{S-65}, we have $\|\psi_{\varepsilon,\beta} \|_{W^{1,p}(\Omega)}$ is uniformly bounded for every $1\leq p<2$ and bounded open subset $\Omega \subset \R^2$. Hence, by Rellich-Kondrakov theorem, $\psi_{\epsilon,\beta}$ converges to $\psi_{\beta} \in W^{1,p}_{\loc}(\R^2)$ pointwise and in $L^2_{\loc}(\R^2)$, up to a subsequence, for every $1 \leq p<2$ as $\varepsilon$ converged to zero. Moreover, by \eqref{eq:approxequanegabetageneral}, $\|e^{\psi_{\varepsilon,\beta}}\|_{L^{\infty}(\Omega)}$ is uniformly bounded on bounded subsets $\Omega \subset \R^2$ and $V_{\varepsilon}$ converges strongly in $L^1_{\loc}(\R^2)$ to $V$. In conclusion, by $\beta> -\alpha(V)$, \eqref{eq:lowerboundpsi1}, \eqref{eq:approxequanegabetageneral}, and Lebesgue's dominated convergence, we have
\begin{equation}
\label{eq:approxequanegabeta}
\begin{aligned}
& \psi_{\beta} \leq -2\beta \log( |\bx|+1)+\log |\beta| +C_2, \quad \textup{a.e. in } \R^2,\\
   & -\Delta \psi_{\beta} =4 \pi \beta \,V e^{\psi_{\beta}}, \quad \textup{weakly in } \R^2,\\
   & \int_{\R^2} V e^{\psi_{\beta}} =1.
\end{aligned}
\end{equation}
Now, we claim that \begin{align*}
    \psi_{\beta_2} + \log|\beta_2| < \psi_{\beta_1} + \log|\beta_1|, \quad \textup{a.e. in } \R^2,
\end{align*}
for every $\beta_2 > \beta_1 > -\alpha(V)$. This is followed by an application of the strong maximum principle; see \cite[Ch. 6.4.1, Thm. 2; Ch. 6.4.2, Thm. 4]{E-10}. To see this, let $\beta_2 > \beta_1 > -\alpha(V)$. Now, it is sufficient to prove that 
\begin{align*}
   \phi_2:= \psi_{\varepsilon,\beta_2} + \log|\beta_2| <\phi_1 := \psi_{\varepsilon,\beta_1} + \log|\beta_1|, \quad \textup{in } \R^2,
\end{align*}
for every $0<\varepsilon<1$. To prove this, we use \eqref{eq:approxequanegabetageneral} and \eqref{eq:assymapproxsolepsilon} to derive that,
for $R>0$ large enough and  $\phi := \phi_1- \phi_2$,
\begin{align*}
   \phi>0, \quad \textup{in } \partial D(0,R),
\end{align*}
and 
\begin{align*}
  -\Delta \phi +f  \phi \geq 0, \quad \textup{in } D(0,R),
\end{align*}
where
\begin{align*}
    f := 4\pi V_{\varepsilon} \frac{e^{\phi_1}-e^{\phi_2}}{\phi_1- \phi_2} \geq 0, \quad \textup{in } \R^2.
\end{align*}
Then, by the strong maximum principle; see \cite[Ch. 6.4.1, Thm. 2; Ch. 6.4.2, Thm. 4]{E-10}, we have 
\begin{align*}
    \phi>0, \quad \textup{in } D(0,R).
\end{align*}
By taking $R \to \infty$, we conclude the claim. Finally, if $\alpha(V) <\infty,$ we have $\psi_{\beta}+ \log |\beta|$ increases as $\beta>\alpha(V)$ decreases to $\alpha(V)$ and 
\begin{align*}
    \psi_{\beta}(\bx) \leq -2\alpha(V) \log( |\bx|+1)+\log |\alpha(V)| +C_2, \quad \textup{for every }  \bx \in \R^2.
\end{align*}
Then, by using \eqref{eq:approxequanegabeta} and a standard $L^1$-elliptic regularity result; see \cite{S-65}, as before, $\psi_{\beta}$ converges pointwise and in $L^2_{\loc}(\R^2)$ to $\psi_{\alpha(V)}$ as $\beta$ decrease to $\alpha(V)$. In conclusion, by using \eqref{eq:approxequanegabeta} again, together with the monotone convergence theorem, we conclude that 
\begin{align*}
    & \psi_{\alpha(V)} \leq -2 \alpha(V) \log( |\bx|+1)+\log |\alpha(V)| +C_2, \quad \textup{a.e. in } \R^2,\\
   & -\Delta \psi_{\alpha(V)} =4 \pi \alpha(V)\,  V e^{\psi_{\alpha(V)}}, \quad \textup{weakly in } \R^2,\\
   & \int_{\R^2} V e^{\psi_{\alpha(V)}} =1,
\end{align*}
which completes the proof,

\end{proof}

\section{Proof of uniqueness of solutions to Liouville equation}
In this section, we prove Theorem \ref{thm:comparison} and Theorem \ref{thm:uniquenessposistivbeta}.
\begin{proof}[Proof of Theorem \ref{thm:comparison}]
Let $$\inf_{\R^2} \psi_1 - \psi_2 <0.$$ We try to drive a contradiction. First, assume that $\psi_1(0) > \psi_2(0)$. Then, by continuity of $\psi_1,\psi_2$, there exists $\varepsilon>0$ such that $\psi_1 > \psi_2$ on $\overline{D(0,\varepsilon)}$. By scaling the equations of $\psi_1,\psi_2,$ we can assume that $\varepsilon >2.$ Then, by the asymptotic properties of $\psi_1,\psi_2$; see \eqref{eq:aymptoticcondition}, there exists $x_m \in \R^2 \setminus \overline{D(0,\varepsilon)}$ such that
\begin{align}
   \frac{\psi_1(\bx_m) - \psi_2(\bx_m)}{\log(|\bx_m|-1)} = \inf_{\bx \in \R^2 \setminus D(0,2)} \frac{\psi_1(\bx) - \psi_2(\bx)}{\log(|\bx|-1)} < 0.
\end{align}
Hence,
\begin{align*}
   &(i).\quad \psi_1(\bx_m) - \psi_2(\bx_m) <0, \\& (ii).\quad \nabla \left( \frac{\psi_1(\bx) -\psi_2(\bx)}{\log(|\bx|-1)} \right)\bigg |_{\bx=\bx_m} =0,\\&(iii).\quad\Delta \left( \frac{\psi_1(\bx) -\psi_2(\bx)}{\log(|\bx|-1)} \right)\bigg |_{\bx=\bx_m} \geq 0.
\end{align*}
By $(ii)$, we obtain
\begin{align*}
    \nabla (\psi_1 -\psi_2)(\bx_m) = -(\psi_1 - \psi_2)(\bx_m) \log(|\bx_m|-1) \nabla \left( \frac{1}{\log(|\bx|-1)} \right) \bigg |_{\bx=\bx_m}.
\end{align*}
Hence, by using $(i), (iii)$, we derive that 
\begin{align*}
    0 &\leq  \Delta \left( \frac{\psi_1(\bx) -\psi_2(\bx)}{\log(|\bx|-1)} \right)\bigg |_{\bx=\bx_m}  \\&=  \frac{\Delta (\psi_1(\bx) -\psi_2(\bx))}{\log(|\bx|-1)}   \bigg |_{\bx=\bx_m} \\&+  (\psi_1(\bx) -\psi_2(\bx)) \,\Delta \left(\frac{1}{\log(|\bx|-1)}\right) \bigg |_{\bx=\bx_m} \\ &+ 2 \nabla (\psi_1(\bx) -\psi_2(\bx)) \cdot \nabla \left(\frac{1}{\log(|\bx|-1)}\right) \bigg |_{\bx=\bx_m}
    \\ &\leq  \frac{  F_1(\bx_m,\psi_1(\bx_m)) - F_2(\bx_m,\psi_2(\bx_m))}{\log(|\bx_m|-1)}\\ &+  \frac{ (\psi_1 -\psi_2)(\bx_m)}{|\bx_m| (|\bx_m|-1)^2(\log(|\bx_m|-1))^2 } < 0,
  \end{align*}
  which is a contradiction. Note that, we used \eqref{eq:comparisonassum} on the last inequality above. Now, assume that $\psi_1(0) = \psi_2(0)$. Then, $\psi_{\varepsilon} := \psi_1 + \varepsilon$ for every $\varepsilon>0$ satisfies
\begin{align*}
    -\Delta \psi_{\varepsilon}(\bx) +F_{\varepsilon}(\bx,\psi_{\varepsilon}(\bx)) \geq -\Delta \psi_2(\bx) +F_2(\bx,\psi_2(\bx)),
\end{align*}
for $\bx \in \R^2$, where $F_{\varepsilon}(\bx,r)=  F_1(\bx,r-\varepsilon)$ for $\bx \in \R^n, r \in \R.$ Moreover, $\psi_{\varepsilon} (0) > \psi_2(0)$,
 \begin{align*}
        F_{2}(\bx,r_2) \geq F_{\varepsilon}(\bx,r_1),    \end{align*}
for every $\bx \in \R^2, r_2 \geq r_1,$ and
\begin{align*}
    \liminf_{\bx \to \infty} \frac{\psi_{\varepsilon}(\bx) - \psi_2(\bx)}{ \log|\bx|} \geq 0.
\end{align*}
  In conclusion, by the previous case, we deduce that $\psi_{\varepsilon} \geq \psi_2$ in $\R^2$. Now, letting $\varepsilon$ converge to zero, we derive that $\psi_1 \geq \psi_2$ in $\R^2$. The final case is that $\psi_1(0) < \psi_2(0)$. Now, let $\bx_0 \in \R^2$ be a fixed point. Define the functions $\phi_1(\bx) := \psi_1(\bx+\bx_0)$ and $\phi_2(\bx) := \psi_2(\bx+\bx_0)$ for $\bx \in \R^2$. Then, 
\begin{align*}
    -\Delta \phi_2(\bx) + \widetilde{F}_2(\bx,\phi_2(\bx)) \leq -\Delta \phi_1(\bx) + \widetilde{F}_1(\bx,\phi_1(\bx)),
\end{align*}
  for every $\bx \in \R^2$, where $\widetilde{F}_i(\bx,r)= F_i(\bx+\bx_0,r)$ for $\bx \in \R^2, r \in \R,i=1,2.$ Since
  \begin{align*}
      \widetilde{F}_2(\bx,r_2) \geq \widetilde{F}_1(\bx,r_1),
  \end{align*}
  for every $\bx \in \R^2, r_1,r_2 \in \R$ satisfying $r_2 \geq r_1$,
  and 
\begin{align*}
    \liminf_{\bx \to \infty}  \frac{\phi_1(\bx) - \phi_2(\bx)}{ \log|\bx|} \geq 0, 
\end{align*}
by the previous steps, we conclude that either $\phi_1 \geq \phi_2$ in $\R^2$ or $\phi_1(0) < \phi_2(0).$ Hence, by $\phi_1(-\bx_0)=\psi_1(0) < \psi_2(0)= \phi_2(-\bx_0)$, we obtain that $$ \psi_1(\bx_0)=\phi_1(0)  < \phi_2(0) = \psi_2(\bx_0).$$ Since $\bx_0$ is arbitrary, we conclude that 
\begin{equation}
\label{eq:initialcomparison}
\begin{aligned}
    &\psi_1 < \psi_2,\quad \textup{in } \R^2,\\
    &\Delta (\psi_2 -\psi_1) \geq 0,\quad \textup{in } \R^2,\\
   & \lim_{x \to \infty}  \frac{\psi_1(\bx)-\psi_2(\bx)}{\log|\bx|} =0.
\end{aligned}
\end{equation}
 Now, take arbitrary $\delta >0, \varepsilon>0,$ and $\bx_0 \in \R^2$. Then, by \eqref{eq:initialcomparison}, we have
\begin{align*}
    \Delta (\psi_2 -\psi_1 - \varepsilon \log(|\bx-\bx_0|-\delta)) \geq \frac{\delta \varepsilon }{|\bx-\bx_0|(|\bx-\bx_0|-\delta)^2}>0,
\end{align*}
for every $\bx \in \R^2 \setminus D(\bx_0,2\delta)$. Note that by \eqref{eq:initialcomparison}, we get 
\begin{align*}
    \lim_{\bx \to \infty} \psi_2(\bx) -\psi_1(\bx) - \varepsilon \log(|\bx|-\delta) = - \infty.
\end{align*}
Hence, by the maximum principle, we derive that 
\begin{align*}
   \sup_{\R^2 \setminus D(\bx_0,2\delta)} \psi_2 -\psi_1 - \varepsilon \log(|\bx-\bx_0|-\delta) = \sup_{\partial D(\bx_0,2\delta)} \psi_2 -\psi_1 - \varepsilon \log(|\bx-\bx_0|-\delta).
\end{align*}
In conclusion, by letting first $\varepsilon \to 0$ and then $\delta \to 0$, we arrive at 
\begin{align*}
    \sup_{\R^2} \psi_2 -\psi_1 =  \psi_2(\bx_0) -\psi_1(\bx_0).
\end{align*}
Since $\bx_0$ is arbitrary, $\psi_2-\psi_1$ is a constant. Moreover, by \eqref{eq:comparisonassum} and \eqref{eq:initialcomparison}, we obtain
\begin{align*}
    F_1(\bx,\psi_1(\bx)) = F_2(\bx,\psi_2(\bx)), \quad \textup{for every } \bx \in  \R^2.
\end{align*}

\end{proof}

\begin{example}
\label{ex:keyexample}
    We consider the case that $V(\bx) := e^{-|\bx|^2}, V_{\theta}(\bx) := e^{-|\bx|^2 + \theta x_1}$ for $\bx =(x_1,x_2) \in \R^2$ and $\theta \in \R.$ Then, by the exponential decay of $V_{\theta}$, Lemma \ref{lem:degreeformula}, and Corollary \ref{thm:strongestnegativebeta}, there exists $\psi_{\theta} \in C^{\infty}(\R^2)$ which satisfies 
    \begin{equation}
    \label{eq:exampleequatio}
    \begin{aligned}
&-\Delta \psi_{\theta} = -4 \pi  V_{\theta}\, e^{\psi_{\theta}}, \quad \textup{weakly in } \R^2, \\
&\int_{\R^2} V_{\theta}\, e^{\psi_{\theta}} =1,\\
& \lim_{\bx \to \infty} \frac{\psi_{\theta}(\bx)}{\log|\bx|} = 2.
    \end{aligned}
    \end{equation}
    Then, $\widetilde{\psi}_{\theta} = \psi_{\theta} + \theta x_1$ is a solution to 
\begin{equation}
\label{eq:differentassymsolex}
    \begin{aligned}
&-\Delta \widetilde{\psi}_{\theta} = -4 \pi  V\, e^{\widetilde{\psi}_{\theta}}, \quad \textup{in } \R^2,\\
&\int_{\R^2} V\, e^{\widetilde{\psi}_{\theta}} =1.
    \end{aligned}
    \end{equation}
Moreover, 
\begin{align*}
    \widetilde{\psi}_{\theta}(\bx) =\theta x_1  + 2 \log |\bx| + o(\log |\bx|), \quad \bx \in \R^2 \setminus D(0,1).
\end{align*}
Hence, by changing $\theta$, we get a parameter of different solutions to \eqref{eq:differentassymsolex}. 
    
\end{example} 

Finally, we bring the proof of Theorem \ref{thm:uniquenessposistivbeta}.
\begin{proof}[Proof of Theorem \ref{thm:uniquenessposistivbeta}]
Similar to \cite{Lin-00}, we define $\psi(r,s)$ for every $r\geq 0,s \in \R$ as the unique solution of the ODE
\begin{align}
\label{eq:equationpsiode}
   \partial_r^2  \psi(r,s) + \frac{1}{r} \partial_r \psi(r,s) +  r^n V(r) e^{\psi(r,s)} = 0,
\end{align}
    with the initial data 
\begin{align*}
    \psi(0,s) =s , \quad \partial_r  \psi(0,s) =0.
\end{align*}
Then, by integrating the equation for $ \psi(r,s)$, we derive that $\psi(r,s)$ is decreasing in $r$ and
\begin{align*}
   \lim_{r \to \infty} -r \partial_r \psi(r,s) = 2\beta(s) := \int_{0}^{\infty} r^{n+1} V(r) e^{\psi(r,s)} \, \dd r.
\end{align*}
Note that since $\psi(r,s)$ is decreasing in $r$ and \eqref{eq:uniformpointwisebound}, \eqref{eq:radialLiouvillesol} hold, we obtain that $\beta(s)$ is positive and finite. Moreover,
by Corollary \ref{cor:existencepositiveassymdecay}, we have
\begin{align}
\label{eq:betascondition}
   n > \beta(s)-2.
\end{align}
Now, by taking $\phi(r,s) = \partial_s \psi(r,s)$, we derive the linear equation
\begin{align}
\label{eq:ODEforphi}
    \partial_r^2  \phi(r,s) + \frac{1}{r} \partial_r \phi(r,s) +  r^n V(r) e^{\psi(r,s)} \phi(r,s) = 0,
\end{align}
with the initial data
\begin{align}
\label{eq:initialdataforphi}
     \phi(0,s) =1 , \quad \partial_r  \phi(0,s) =0.
\end{align}
We first prove that for a constant $C$, we have 
\begin{align}
\label{eq:boundonphi}
    |\phi(r,s)| \leq C \log(r+2),
\end{align}
for every $r \geq 0$ and $s$ being at a fixed neighbourhood of a fixed point $s_0 \geq 0.$
  To prove this, for $r_0 >1$, we use the integral of the equation for $\phi$, \eqref{eq:ODEforphi}, to obtain
\begin{align*}
   \sup_{0 \leq t \leq r} |t\, \partial_t \phi(t,s)| 
   &\leq \int_{0}^r t^{n+1} V(t) e^{\psi(t,s)} |\phi(t,s)| \, \dd t\\& \leq \int_{r_0}^r t^{n+1} V(t) e^{\psi(t,s)} |\phi(t,s)-\phi(0,s)| \, \dd t \\&+|\phi(0,s)| \int_{r_0}^r t^{n+1} V(t) e^{\psi(t,s)}  \, \dd t + \int_{0}^{r_0} t^{n+1} V(t) e^{\psi(t,s)} |\phi(t,s)| \, \dd t
   \\ & \leq e^s \sup_{0 \leq t \leq r} |t\, \partial_t \phi(t,s)| \int_{r_0}^r t^{n+1} V(t)  \, \dd t + e^s \int_{r_0}^r t^{n+1} V(t)  \, \dd t \\&+ e^s \int_{0}^{r_0} t^{n+1} V(t)  |\phi(t,s)| \, \dd t,
\end{align*}
where on the last inequality we used that $\psi(t,s)$ is decreasing in $t$ and \eqref{eq:ODEforphi}. Since \eqref{eq:uniformpointwisebound} holds, by taking $r_0$ large enough, we obtain that 
\begin{align}
\label{eq:derivativephiest}
     \sup_{0 \leq t \leq r} |t\, \partial_t \phi(t,s)|  \leq C,
\end{align}
 for $r\geq 0$ and $s$ in a fixed neighbourhood of $s_0 \geq 0$, where $C$ is a constant independent of $r$. This completes the proof of \eqref{eq:boundonphi}. Hence, by \eqref{eq:uniformpointwisebound}, \eqref{eq:boundonphi}, and $\psi$ being decreasing, we imply $\beta(s) \in C^1(\R)$. Now, to prove the uniqueness, it is sufficient to show that there exists no point $s_0 \in \R$ such that $\beta'(s_0)=0.$ Assume that $\beta'(s_0)=0$ for some $s_0 \in \R.$ Let us define $\phi(r) :=\phi(r,s_0), \psi(r) := \psi(r,s_0), \beta :=\beta(s_0)$ for every $r \geq 0$. We claim that $\phi$ is a bounded function. To see this, for $r \geq 1$, we use the equation for $\phi$, \eqref{eq:ODEforphi}, and $\beta'(s_0)=0$ to obtain that 
 \begin{align*}
     r |\phi'(r)| &= \int_r^{\infty} t^{n+1} V(t) e^{\psi(t)} |\phi(t)| \leq  C\int_r^{\infty} t^{n+1} V(t) e^{\psi(t)}  \log(t+2) 
     \\ & \leq \frac{C e^{s_0}}{r^{\delta/2}}  \sup_{t \geq 0} t^{n+2+\delta} V(t)  \int_r^{\infty}   \frac{ \log(t+2)}{t^{1+\delta/2}} \, \dd t\leq \frac{C'}{r^{\delta/2}},
 \end{align*}
where $C,C'$ are constants independent of $r$. Here, we used \eqref{eq:boundonphi} on the first inequality and  \eqref{eq:uniformpointwisebound} on the last inequality. Hence,
\begin{align*}
    |\phi(r) -\phi(1)| \leq \int_1^r |\phi'(t)| \, \dd t \leq \int_1^{\infty} \frac{C'}{t^{1+\delta/2}} \, \dd t =  \frac{2 C'}{\delta},
\end{align*}
which completes the proof of the claim. Now, similar to \cite{Lin-00}, we define the functions 
\begin{equation}
\label{eq:defofP}
\begin{aligned}
    &P(\psi)(r,s) := r \partial_r\psi(r,s) \left(\frac{1}{2} r \partial_r\psi(r,s) + \beta(s)\right)  + r^{n+2} V(r) e^{\psi},\\
    &P(\phi)(r,s) := \partial_s P(\psi)(r,s)=  r \partial_r\phi(r,s) (r \partial_r \psi(r,s) + \beta(s))\\&+\beta'(s)\, r\, \partial_r\psi(r,s)   + r^{n+2} V e^{\psi(r,s)} \phi(r,s). 
\end{aligned}
\end{equation}
Then, by \eqref{eq:equationpsiode}, we have
\begin{align*}
    \partial_r  P(\psi)(r,s) &= \partial_r(r \partial_r\psi(r,s)) ( r \partial_r\psi(r,s) + \beta(s))  + \partial_r\left(r^{n+2} V(r) e^{\psi(r,s)}\right) 
    \\ &=-r^{n+1} V(r) e^{\psi(r,s)} ( r \partial_r\psi(r,s) + \beta(s))  + \partial_r\left(r^{n+2} V(r) e^{\psi(r,s)}\right)
    \\ &=(r  V'(r)+  (n+2 -\beta) V(r)) r^{n+1} e^{\psi(r,s)}.
\end{align*}
Hence,
\begin{align}
\label{eq:integralformulationofP}
     P(\psi)(r,s) = \int_0^r (r  V'(r)+  (n+2 -\beta(s)) V(r)) r^{n+1} e^{\psi(r,s)} \, \dd r.
\end{align}
Now, we prove that 
\begin{align}
\label{eq:positivityofP}
     P(\psi)(r,s) \geq 0, \quad r \geq 0.
\end{align}
 Note that, since $\lim_{r \to \infty} r  \partial_r\psi(r,s) =-2\beta(s)$ and \eqref{eq:uniformpointwisebound}, we obtain 
 \begin{align}
     \label{eq:limitofPpsi}
     \lim_{r \to \infty} P(\psi)(r,s) =0.
 \end{align}
  Moreover, by \eqref{eq:unqiuenesscondition} and \eqref{eq:betascondition}, there exists $r_0>0$ such that 
 \begin{equation*}
\begin{aligned}
(n+2-\beta(s)) V(r) + r V'(r) \geq 0, \quad r \leq r_0,\\
(n+2-\beta(s)) V(r) + r V'(r) \leq 0,\quad r \geq r_0.
\end{aligned}
\end{equation*}
Hence, by \eqref{eq:integralformulationofP}, we derive that $ P(\psi)(r,s) \geq 0$ for every $r \leq r_0$ and 
\begin{align*}
    P(\psi)(r,s) \geq \int_0^{\infty} (r V'(r)+  (n+2 -\beta(s)) V(r)) r^{n+1} e^{\psi(r,s)} \, \dd r = \lim_{r \to \infty}  P(\psi)(r,s) =0,
\end{align*}
for every $r \geq r_0$, which concludes the proof of \eqref{eq:positivityofP}. Now, differentiating \eqref{eq:integralformulationofP} with respect to $s$ at $s=s_0$, we derive that 
\begin{equation}
\label{eq:integralformulaphi}
\begin{aligned}
    P(\phi)(r,s_0) &=  \int_0^r (r V'(r)+  (n+2 -\beta V(r)) r^{n+1} \phi(r) e^{\psi(r)}
    \\&= \int_0^r  \phi(r) \partial_r  P(\psi)(r,s_0) 
    \\ &= \phi(r) P(\psi)(r,s_0)  - \int_0^r  
  P(\psi)(r,s_0) \,\partial_r \phi(r).  
\end{aligned}
\end{equation}
 Now, if $\phi(r)$ has at least two zeros, then, by \eqref{eq:ODEforphi} and \eqref{eq:initialdataforphi}, there exists a point $r_1 >0$ such that $\phi'(r) \leq 0$ for $r \leq r_1$, $\phi(r_1)<0,$ $\phi'(r_1)=0.$ Hence, by using \eqref{eq:positivityofP}, we obtain
 \begin{equation*}
      P(\phi)(r_1,s_0) \geq \phi(r_1) P(\psi)(r_1,s_0).
 \end{equation*}
In conclusion, by \eqref{eq:defofP}, we arrive at
\begin{align*}
    r_1 \psi'(r_1)  \left( \frac{1}{2} r_1 \psi'(r_1) + \beta \right) \geq 0.
\end{align*}
This is a contradiction as
\begin{align*}
  - 2\beta=-\int_{0}^{\infty} r^{n+1} V(r) e^{\psi(r)} \, \dd r <r_1 \psi'(r_1) = -\int_{0}^{r_1} r^{n+1} V(r) e^{\psi(r)} \, \dd r<0.
\end{align*}
Now, if $\phi$ is positive, we obtain the following contradiction
\begin{align}
\label{eq:derivativebetacondition}
    0 =-2\beta'(s_0)= \lim_{r \to \infty} r \phi'(r) = -\int_0^{\infty} r^{n+1} V(r) e^{\psi(r)} \phi(r).
 \end{align}
 Finally, if $\phi$ changes sign only once, then
 \begin{align*}
    \int_0^{r} t^{n+1} V(t) e^{\psi(t)} \phi(t) \, \dd t >0,
 \end{align*}
 by \eqref{eq:derivativebetacondition}, for all $r>0$. Hence, $\phi'(r) \leq 0$ for every $r \geq 0$ and
 \begin{align*}
     0 &= \lim_{r \to \infty} P(\phi)(r,s_0)
     \\&=  \lim_{r \to \infty} \phi(r) P(\psi)(r,s_0) - \int_0^{\infty}  
  P(\psi)(r,s_0) \partial_r \phi(r) \, \dd r
  \\ & = - \int_0^{\infty}  
  P(\psi)(r,s_0) \partial_r \phi(r) \, \dd r.
 \end{align*}
Here, we used \eqref{eq:derivativephiest} and \eqref{eq:defofP} on the first equality, \eqref{eq:integralformulaphi} on the second equality, and \eqref{eq:derivativebetacondition} and the boundedness of $\phi$, together with \eqref{eq:limitofPpsi}, on the last equality. However, by \eqref{eq:positivityofP} and  $\phi'(r) \leq 0$ for every $r \geq 0$, we conclude $\phi$ must be constant which is a contradiction.
 
\end{proof}

\section{Applications}

\label{sec:application}

\subsection{Berger-Nirenberg problem}
The problem of finding Gaussian curvatures $K$ on two-dimensional unit sphere $(S^2,g)$ with metric $g$  which is pointwise conformally equivalent to the Euclidean metric is equivalent to the existence of solutions to the equation
\begin{align*}
    -\Delta u = K e^{2u} - 2,
\end{align*}
in $S^2.$ By applying the stereographic projection map $\phi : S^2 \to \R^2$, we need to find solutions to
\begin{align*}
    -\Delta \psi =8\pi\, V e^{\psi}, \quad \textup{weakly in } \R^2,
\end{align*} where $\psi = 2 u \circ \phi^{-1} ,V = \frac{1}{4\pi} K \circ \phi^{-1}$ and satisfies $\int_{\R^2}|V| e^{\psi}< \infty$ and $\int_{\R^2}V e^{\psi} =1.$ 
Theorem \ref{thm:maintheorem} confirms the existence of such solutions if $V$ is radially symmetric and satisfies \eqref{eq:conditionondelta}.
This condition is equivalent with \cite[Thm. 2]{H-86} if $K$ is a smooth function. Note that in the case that $K$ is H\"older continuous, $K$ being zero around the origin (around the infinity) implies that $|V(\bx)| \leq C |\bx|^{\delta}$ ($|V(\bx)| \leq C (1+|\bx|)^{-\delta}$) for some positive constants $C,\delta.$ Hence, we derive the following result:
\begin{corollary}
    Let $S,N$ be the south and north pole of the two-dimensional unit sphere $(S^2,g)$ where its Gaussian curvature $K$ is H\"older continuous and radially symmetric around $S$. Then, $(S^2,g)$ is conformally equivalent with the Euclidean sphere if $K$ is positive somewhere in $S^2$ and
    \begin{align*}
        \max(K(S),K(N)) \leq 0.
    \end{align*}
\end{corollary} We can generalize this result to the case in which we allow for conical singularities. We remind the reader that a point $p$ on $(S^2,g)$ is a conical singularity of order $\beta$ for the metric $g$ if there exists a conformal map $z : U \to C$ defined on a neighborhood $U$ of $p$, which satisfies $z(p) =0 $ and $g = \psi(z) |z|^{2\beta} |\dd z|^2$ on $U$ for some continuous positive function $\psi.$ If the metric $g$ has conical singularities at the points $p_1, p_2, \cdot \cdot \cdot, p_n \in S^2$ of order $\beta_1, \beta_2, \cdot \cdot \cdot,\beta_n$, we say that $g$ represents the divisor $\pmb{\beta} = \beta_1 p_1 + \cdot \cdot \cdot  + \beta_n p_n.$ Now, by using stereographic projection as before and Theorem \ref{thm:maintheorem} (or Corollary \ref{cor:assymptotic}), we derive the following generalization of \cite[Thm. 5]{T-91}.
\begin{corollary}
    Let $(S^2,g)$ be with the divisor $\pmb{\beta} = \beta_1 S + \beta_2 N $, where $S,N$ are the south and north pole, respectively, and $\beta_1, \beta_2 \in \R$. Assume that the Gaussian curvature $K$ of $S^2$ is H\"older continuous and radially symmetric around $S$, which is positive somewhere in $S^2$ and satisfies 
    \begin{align*}
        K^+(z_1) \leq C |z_1 -z_1(0)|^{l_1},\\
         K^-(z_2) \leq C |z_2 -z_2(0)|^{l_2},
    \end{align*}
   where $C$ is a constant and $z_1: \C \to S^2,z_2:\C \to S^2$ are coordinates with with $z_1(0)= S,z_2(0)=N.$ Then, there exists a conformal metric on $(S^2,g)$ if 
\begin{align*}
-1<\beta_2< \beta_1, \quad l_1 >-(\beta_1+ \beta_2), \quad l_2 > -2(\beta_1+ \beta_2)-2.
\end{align*}

\end{corollary}

\subsection{Blow-up analysis}
Motivated by the question of Brezis and Merle in \cite{BM-91}, we prove the following results for $\beta <0$ and $\beta>0$, respectively.
\begin{corollary}
\label{cor:blowuplimitnegativebeta}
    Let $\beta<0 , V_n \in C^{0,\alpha}(\R^2)$ be a sequence of non-negative functions and $\psi_n \in C^{2,\alpha}(\R^2)$ be a sequence of solutions to 
\begin{align*}
    -\Delta \psi_n =4\pi \beta V_n e^{\psi_n},\quad \textup{weakly in } \R^2, 
 \end{align*}
which satisfies 
\begin{align*}
  \lim_{\bx \to \infty} \frac{\psi_n(\bx)}{\log|\bx|} = -2\beta.
\end{align*}
Assume that $V_n$ converges to $V$ locally uniformly and there exist $W_1,W_2 \in C^{0,\alpha}(\R^2), \phi_1,\phi_2 \in C^{2,\alpha}(\R^2)$ such that 
\begin{align*}
   &0 \leq W_1 \leq V_n \leq W_2,\quad \textup{in } \R^2,\\
    -&\Delta \phi_i =4\pi \beta \, W_i e^{\phi_i},\quad \textup{in } \R^2, \\
    & \lim_{\bx \to \infty} \frac{\phi_i(\bx)}{\log|\bx|} = -2\beta,
\end{align*}
for every $i=1,2.$ Then, $\psi_n$ converges to $\psi$ locally uniformly such that 
\begin{align*}
    &-\Delta \psi =4\pi \beta\, V e^{\psi},\quad \textup{weakly in } \R^2, \\
    & \lim_{\bx \to \infty} \frac{\psi(\bx)}{\log|\bx|} = -2\beta.
 \end{align*}
\end{corollary}

\begin{proof}
    By Theorem \ref{thm:comparison}, we get 
\begin{align*}
  \phi_2 \leq  \psi_n \leq \phi_1, \quad \textup{uniformly in } \R^2.
\end{align*}
Hence, by a standard elliptic regularity result; \cite[Ch. 9]{GT-01}, $\|\psi_n\|_{C^{2,\alpha}(\Omega)}$ is bounded for every bounded subset $\Omega \subset \R^2$. In conclusion, by Rellich-Kondrakov theorem, $\psi_n$ converges to $\psi$ pointwise and in $L^2_{\loc}(\R^2)$, up to a subsequence. Moreover,
\begin{equation}
\begin{aligned}
-\Delta& \psi =4\pi \beta V e^{\psi},\quad \textup{weakly in } \R^2,\\
 &   \phi_2 \leq  \psi \leq \phi_1, \quad \textup{in } \R^2.
\end{aligned}
\end{equation}
 In conclusion,
\begin{align*}
     \lim_{\bx \to \infty} \frac{\psi(\bx)}{\log|\bx|} =   \lim_{\bx \to \infty} \frac{\phi_1(\bx)}{\log|\bx|} =   \lim_{\bx \to \infty} \frac{\phi_2(\bx)}{\log|\bx|} = -2\beta,
\end{align*}
which completes the proof.

\end{proof}

\begin{corollary}
\label{cor:blowuplimitpositivebeta}
      Let $\beta_k>0$ converge to $\beta >0$, $\delta_k>0$ converge to $\delta>0$, and $V_k \in  L^1_{\loc}(\R^2)$ be a sequence of radially symmetric functions that converge to $V \in L^1_{\loc}(\R^2)$ pointwise and in $L^1_{\loc}(\R^2)$. Assume that, for some $\varepsilon>0$, we have 
      \begin{equation}\begin{aligned}
    \label{eq:assymptoicpropeorysequecne}
   & \int_{D(0,1)} |V(\bx)| \, |\bx|^{-\varepsilon} \, \dd \bx + \int_{\R^2 \setminus D(0,1)} |V(\bx)| \, |\bx|^{-2\beta+ \varepsilon} \, \dd \bx < \infty,\\
   & \lim_{k \to \infty}\int_{D(0,1)} |V_k(\bx)-V(\bx)| \, |\bx|^{-\varepsilon} \, \dd \bx + \int_{\R^2 \setminus D(0,1)} \left|V_k(\bx) |\bx|^{-2\beta_k+ \varepsilon}-V(\bx) |\bx|^{-2\beta+ \varepsilon}\right| \,  \dd \bx  = 0,\\
&\sup_k\int_{D(0,1)} V_k^+(\bx) \,|\bx|^{-\beta_k-\delta_k} \, \dd \bx + \int_{\R^2 \setminus D(0,1)} V_k^+(\bx) |\bx|^{-\beta_k+\delta_k} \, \dd \bx < \infty.
\end{aligned}
\end{equation}
 and $\psi_k \in H^1_{\loc}(\R^2)$ is the radially symmetric solutions constructed in Theorem \ref{thm:maintheorem} which satisfies 
\begin{equation}
\label{eq:sequecneblowupass}
\begin{aligned}
   & -\Delta \psi_k = 4 \pi \beta_k  V_k e^{\psi_k},\quad \textup{weakly in } \R^2,\\
    & \int_{\R^2}  |V_k| e^{\psi_k} < \infty, \, \int_{\R^2}  V_k e^{\psi_k} = 1.
\end{aligned}
\end{equation}
Then, $\psi_k$ converges to $\psi \in H^1_{\loc}(\R^2)$ pointwise and in $L^2_{\loc}(\R^2)$, up to a subsequence, which satisfies 
\begin{equation}
\label{eq:limitfunctionequa}
\begin{aligned}
   & -\Delta \psi = 4 \pi \beta  V e^{\psi},\quad \textup{weakly in } \R^2,\\
    & \int_{\R^2}  |V| e^{\psi} < \infty, \, \int_{\R^2}  V e^{\psi} = 1.
\end{aligned}
\end{equation}

\end{corollary}
\begin{proof}
We use Theorem \ref{thm:maintheorem} to obtain 
\begin{align*}
  C_0 := \sup_{k>0} \int_{\R^2} |\nabla (\psi_k + \phi_k)|^2 < \infty,
\end{align*}
where $\phi_k \in C^{\infty}(\R^2)$ are fixed functions, satisfying $\phi_k(r) = 2\beta_k \log r$ for $r \geq 1$. Hence, by Proposition \ref{prop:logbound}, we have
\begin{align*}
   |\psi_k+ \phi_k(r)-\psi_k+ \phi_k(1)|^2 \leq  \frac{C_0}{2\pi}|\log r|, \quad \textup{for every } r >0.
\end{align*}
In conclusion, 
\begin{equation}
\label{eq:upperboundforblowupsequecne}
\begin{aligned}
|\psi_k+ \phi_k(r)|   &= \left|\psi_k+ \phi_k(r) - \log\left( \int_{\R^2} V_k \, e^{-\phi_k} e^{\psi_k + \phi_k}  \right)\right |\\&\leq \left |\psi_k+ \phi_k(r) - \psi_k+ \phi_k(1) \right | +\left |\log\left( \int_{\R^2} V_k \, e^{-\phi_k} e^{\psi_k + \phi_k - \psi_k+ \phi_k(1)}  \right) \right |\\
\\&\leq \varepsilon |\log r| + \frac{C_0}{2\pi \varepsilon} + \left | \int_{\R^2} V_k \, e^{-\phi_k} e^{\varepsilon |\log r| + \frac{C_0}{2\pi \varepsilon}} \right |,
\end{aligned}
\end{equation}
for every $\varepsilon>0$ and $r>0$, where we used \eqref{eq:sequecneblowupass} in the first equality. Notice that, by \eqref{eq:assymptoicpropeorysequecne}, the last term satisfies
\begin{align*}
    \sup_k \left | \int_{\R^2}  V_k e^{-\phi_k} e^{\varepsilon |\log r| + \frac{C_0}{2\pi \varepsilon}} \right | < \infty.
\end{align*}
In particular, $\psi_k$ is uniformly bounded in $H^1(D)$ for every finite disk $D \subset \R^2$ and, by Rellich-Kondrakov theorem, converges to $\psi$ pointwise and locally in $L^2_{\loc}(\R^2)$, up to a subsequence. Hence, 
\begin{align*}
     -\Delta \psi = 4 \pi \beta \, V e^{\psi},\quad \textup{weakly in } \R^2.
\end{align*}
Moreover, by \eqref{eq:assymptoicpropeorysequecne}, \eqref{eq:sequecneblowupass}, \eqref{eq:upperboundforblowupsequecne}, and the Lebesgue's dominated convergence theorem, we obtain
 \begin{align*}
 \int_{\R^2}  V e^{\psi}  = \lim_{k \to \infty} \int_{\R^2}  V_k e^{-\phi_k} e^{\psi_k+\phi_k} =1,
 \end{align*}
 which completes the proof. 

\end{proof}
\subsection{Statistical mechanics}
The statistical mechanics of the point vortex system in the mean field limit has been used as a model to understand the behavior of the 2-dimensional turbulent flows; see \cite{BD-15,BJLY-18,CLMP-92,CLMP-95,CK-94,ES-93, ES-78,K-93, KL-97,Lin-00,O-49,PD-93,SO-90,TY-04,W-92}. In the presence of external fields with a point vortex at the origin, the mean filed equation is as follows:
\begin{align*}
-\Delta \Psi(\bx) = \frac{ |\bx|^{n} e^{-\beta \Psi - \gamma |\bx|^{\alpha} }}{\int_{\R^2} |\by|^{n} e^{-\beta \Psi(\by) - \gamma |\by|^{\alpha} } \, \dd \by}, \quad \textup{for every } \bx \in \R^2,
\end{align*}
 where $\Psi$ is the stream function of the flow, $\gamma,\alpha$ are positive, $n$ is non-negative, and $\beta =-\frac{1}{\kappa T_{\textup{stat}}}$ for statistical temperature $T_{\textup{stat}}$ and Boltzmann constant $\kappa.$ We note that 
 \begin{align*}
   \Psi(\bx) =- \frac{1}{2\pi} \int_{\R^2}  (\log|\bx-\by| - \log (|\by|+1) ) \, \rho(\by) \, \dd \by, \quad \textup{for every } \bx \in \R^2,
 \end{align*}
 where 
 \begin{align}
 \label{eq:densityequation}
  \rho(\bx):= \frac{ |\bx|^{n} e^{-\beta \Psi(\bx) - \gamma |\bx|^{\alpha} }}{\int_{\R^2} |\by|^{n} e^{-\beta \Psi(\by) - \gamma |\by|^{\alpha} } \, \dd \by} \quad \textup{for every } \bx \in \R^2,
 \end{align}
 is the density of the flow, which satisfies $\int_{\R^2} \rho = 1$. By defining a new functions $\psi := -\beta \Psi $, it is enough to solve
\begin{align}
\label{eq:onsagerequation}
    -\Delta \psi(\bx) = -\beta \frac{ |\bx|^{n} e^{- \gamma |\bx|^{\alpha} } e^{ \psi(\bx)}}{\int_{\R^2} |\by|^{n} e^{- \gamma |\by|^{\alpha} } e^{ \psi(\by)}\, \dd \by}, \quad \textup{for every } \bx \in \R^2.
 \end{align}
Now, by taking the equation for $\psi - \log \left(\int_{\R^2} |\by|^{n} e^{- \gamma |\by|^{\alpha} } e^{ \psi(\by)} \, \dd \by \right)$ and using Corollary \ref{cor:existencepositiveassymdecay}, Lemma \ref{lem:degreeformula}, and Theorem \ref{thm:uniquenessposistivbeta}, we obtain the following result.

\begin{corollary}
\label{cor:existenceonsagereq}
    The equation \eqref{eq:onsagerequation} has a unique radially symmetric solution with bounded density $\rho$, defined in \eqref{eq:densityequation}, if $n >\frac{-\beta-8\pi }{4\pi}$. Moreover, if there exists a solution (not necessarily radially symmetric) to \eqref{eq:onsagerequation} with bounded density $\rho$, then $n >\frac{-\beta-8\pi }{4\pi}$ must hold.
\end{corollary}
 The physical meaning is that, for positive statistical temperatures, if the temperature gets closer to zero, then the strength of the vortex at the origin must increase to obtain a solution. 

The next result derives the concentration properties of the measures in \cite{CLMP-92,CLMP-95,BT-02,BT-07,T-08}.

\begin{corollary}
    Let $\gamma_k>0, \alpha_k>0,n_k \geq 0,\beta_k<0$, and $\psi_k$ be the sequence of radially symmetric functions, satisfying 
    \begin{align*}
        -\Delta \psi_k(\bx) = -\beta_k \frac{ |\bx|^{n_k} e^{- \gamma_k |\bx|^{\alpha_k} } e^{ \psi_k(\bx)}}{\int_{\R^2} |\bx|^{n_k} e^{- \gamma_k |\bx|^{\alpha_k} } e^{ \psi_k(\bx)}}, \quad \textup{for every } \bx \in \R^2.
    \end{align*}
    Assume that $\lim_{k \to \infty} \alpha_k = \alpha > 0, \lim_{k \to \infty} \gamma_k = \gamma \geq 0, \lim_{k \to \infty} n_k = n \geq 0, \lim_{k \to \infty} \beta_k = \beta < 0$. Then, up to a subsequence, we have the following possibilities:\\\\
    (i). If
    $n> \frac{-\beta-8\pi }{4\pi} $, then 
    \begin{align}
        \label{eq:newsequecne}
        \phi_k(r) := \psi_k\left(\gamma_k^{-\frac{1}{\alpha_k}} r  \right)- \frac{n_k+2}{\alpha_k} \log (\gamma_k)-\log  \left (\int_{\R^2} |\bx|^{n_k} e^{- \gamma_k |\bx|^{\alpha_k} } e^{ \psi_k(\bx)} \right ),
    \end{align}
     defined for every $r \geq 0$, converge to $\phi \in C^{\infty}(\R^2)$ in $L^{\infty}_{\loc}(\R^2)$ such that
    \begin{align*}
       & -\Delta \phi(\bx) = -\beta |\bx|^{n} e^{- |\bx|^{\alpha} } e^{ \phi(\bx)}, \quad \textup{for every } \bx \in \R^2,\\
      &  \int_{\R^2} |\bx|^{n} e^{- |\bx|^{\alpha} } e^{ \phi(\bx)} \, \dd \bx = 1.
    \end{align*}
    (ii). If $n =\frac{-\beta-8\pi }{4\pi}$, then $\phi_k$ converges locally uniformly to $-\infty$ on compact subsets of $\R^2 \setminus 0$, $\phi_k(0)=\psi_k(0)$ converges to $\infty$, and 
    \begin{align*}
       |\bx|^{n_k} e^{-  |\bx|^{\alpha_k} } e^{ \phi_k(\bx)} \to \delta_0,
    \end{align*}
    in the sense of measures, where $\delta_0$ is the Dirac mass centered at the origin.
\end{corollary}
\begin{proof}

The part $(i)$ is an immediate conclusion of Theorem \ref{thm:uniquenessposistivbeta}, Corollary \ref{cor:blowuplimitpositivebeta}, and a standard elliptic regularity argument; see \cite[Chapter 9]{GT-01}. Now, assume that $n =\frac{-\beta-8\pi }{4\pi}$. We first notice that $\phi_k(r)$ is radially decreasing, and $$\widetilde{\phi}_k(r) := \phi_k(r) + n_k\log r
,$$ satisfies 
\begin{align*}
        -\Delta \widetilde{\phi}_k =f_k:=  -\beta_k e^{- |\bx|^{\alpha_k} } e^{ \widetilde{\phi}_k} -2\pi n_k \delta_0, \quad \textup{weakly in }  \R^2.
\end{align*}
Hence, 
\begin{align}
\label{eq:gradboundsequence}
   r \, |\partial_r \widetilde{\phi_k}(r)| \leq \frac{-\beta_k}{2\pi} + n_k, \quad r > 0.
\end{align}
Assume that $$-\infty<\limsup_{k \to \infty}\widetilde{\phi}_k(1) < \infty.$$ Then, by \eqref{eq:gradboundsequence}, $\widetilde{\phi}_k$ converges to $\widetilde{\phi}$ uniformly on compact subsets of $\R^2 \setminus 0$, up to a subsequence. Since $0<\alpha =\lim_{k \to \infty}\alpha_k$, we can use Prokhorov's theorem to derive a positive measure $\nu$ such that 
\begin{align*}
    \lim_{k \to \infty}  -\beta_k e^{- |\bx|^{\alpha_k} } e^{ \widetilde{\phi}_k(\bx)} - 2\pi n_k \delta_0  = \nu.
\end{align*}
Hence,
\begin{align}
\label{eq:equationofphitilde}
   -\Delta \widetilde{\phi}=  \nu,
\end{align} 
in the sense of distributions. Now, if $\nu(0) \geq 4\pi$, then by the maximum principle 
\begin{align}
\label{eq:boudnfrombelowphi}
    \widetilde{\phi}(\bx) \geq C - 2 \log |\bx|,
\end{align}
for every $\bx$ in a disk $D$ centered at the origin, where $C$ is a constant. Moreover, by Fatou's lemma,
\begin{align*}
 \int_D  e^{- |\bx|^{\alpha} } e^{ \widetilde{\phi}(\bx)} \, \dd \bx  \leq \liminf_{k \to \infty} \int_{\R^2} e^{- |\bx|^{\alpha_k} } e^{ \widetilde{\phi_k}(\bx)} \, \dd \bx =1.
\end{align*}
However, the left-hand side blows up by \eqref{eq:boudnfrombelowphi}. The contradiction implies that $\nu(0) < 4\pi.$ Then, we derive that $f_k$ satisfies $\int_{D}f_k< 4 \pi$ for large enough $k$ and small enough disk $D \subset \R^2$ centered at the origin. Now, let 
\begin{align*}
    C_0 := \sup_{k} |\widetilde{\phi}_k(1)|.
\end{align*}
By the maximum principle, we have $|\widetilde{\phi}_k| \leq g_k + C_0$, where $g_k \in H^1_0(D)$ satisfies 
\begin{align*}
    -\Delta g_k = f_k,\quad \textup{in } D.
\end{align*}
Now, by \cite[Thm. 1]{BM-91} and the fact that $\int_{D}f_k< 4 \pi$, we have 
\begin{align*}
    \int_{D} e^{(1+\delta) |g_k|} \leq C_1,
\end{align*}
for some positive constants $C_1,\delta.$ Since $|\phi_k| \leq g_k + C_0$, we derive that
\begin{align}
       \int_{D} |f_k|^{1+\delta} \leq C_2.
\end{align}
Hence, by \cite[Cor. 4]{BM-91} and Fatou's lemma, we derive that $(\widetilde{\phi}_k)^+$ is uniformly bounded in $D$. Hence, $$\phi_k(\bx) \leq C_3 -n_k \log |\bx|,$$ 
for $\bx \in D$ and a constant $C_3$. Moreover, since $\phi_k $ is radially decreasing, we get 
\begin{align}
    \phi_k(\bx)  \leq C_4,
\end{align}
for $\bx \in \R^2 \setminus D$ and a constant $C_4$.
Then, we can apply Lebesgue's dominated convergence to imply that 
\begin{align*}
    \int_{\R^2}  e^{- |\bx|^{\alpha} } e^{ \widetilde{\phi}(\bx)} \, \dd \bx= 1.
\end{align*}
By defining $\phi(r) := \widetilde{\phi}(r)-n \log r$ for every $r >0$ and using \eqref{eq:equationofphitilde}, we obtain
\begin{align*}
    &-\Delta  \phi(\bx)=  -\beta |\bx|^{n} e^{- |\bx|^{\alpha} } e^{ \phi(\bx)} , \quad \textup{for every }  \bx \in \R^2,\\
     &\int_{\R^2} |\bx|^n  e^{- |\bx|^{\alpha} } e^{ \phi(\bx)} \, \dd \bx = 1.
\end{align*}
Then, by Corollary \ref{cor:existencepositiveassymdecay}, $n>-\frac{\beta}{4\pi} -2$ which is a contradiction. Now, if 
\begin{align*}
      \limsup_{k \to \infty} \widetilde{\phi}_k(1) =  \infty,
\end{align*}
then 
\begin{align*}
   \limsup_{k \to \infty} \phi_k(1)= \infty.
\end{align*}
Since $ \phi_k(r)$ is radially decreasing, we conclude
\begin{align*}
   1&= \limsup_{k \to \infty} \int_{\R^2 }   |\bx|^{n_k} e^{-  |\bx|^{\alpha_k} } e^{ \phi_k(\bx)} \, \dd \bx \\&\geq   \limsup_{k \to \infty} e^{ \phi_k(1)} \int_{\R^2 \setminus D(0,1)}  |\bx|^{n_k} e^{-  |\bx|^{\alpha_k} } \, \dd \bx = \infty,
\end{align*}
which again gives a contradiction. In conclusion, 
\begin{align}
    \lim_{k \to \infty} \widetilde{\phi}_k(1) = - \infty,
\end{align}
or equivalently, by \eqref{eq:gradboundsequence}, we have 
\begin{align*}
   \lim_{k \to \infty} \phi_k  = -\infty,
\end{align*}
 uniformly on the compact subsets of $\R^2 \setminus 0.$ Since $\phi_k $ is radially decreasing, we conclude
  \begin{align}
  \label{eq:diracmassslimit}
       |\bx|^{n_k} e^{-  |\bx|^{\alpha_k} } e^{ \phi_k(\bx)} \to \delta_0,
    \end{align}
in the sense of measures. Finally, if $\liminf_{k \to \infty} \phi_k(0)< \infty$, then 
\begin{align*}
    \int_D  |\bx|^{n_k} e^{-  |\bx|^{\alpha_k} } e^{ \phi_k(\bx)} \, \dd \bx \leq C_0  e^{ \phi_k(0)} \int_D  |\bx|^{n_k} e^{-  |\bx|^{\alpha_k} } \, \dd \bx,
\end{align*}
for every disk $D$. However, $ \int_D  |\bx|^{n_k} e^{-  |\bx|^{\alpha_k} } e^{ \phi_k(\bx)} \, \dd \bx$ becomes small uniformly as $D$ shrinks to the origin. This is a contradiction with \eqref{eq:diracmassslimit}. Therefore, $\lim_{k \to \infty} \phi_k(0)= \infty,$ which completes the proof.
\end{proof}

There exists also an interest in considering the mean field equation on the sphere, which is known as the spherical Onsager equation. We refer to \cite{CK-94,Lin-00,PD-93} for motivations and partial results in this direction. Let $n\geq 0,\beta>0, l<0,\gamma \geq 0.$ If there exists a vortex of density at the south pole (or equivalently north pole), by stereographic transformation, we derive the equation 
\begin{align*}
    \Delta u + K e^{2u} =\pi n \delta_0,
\end{align*}
 in $\R^2$ in the sense of distributions, where 
\begin{align*}
    K(\bx) =(1+|\bx|^2)^l e^{\frac{2 \gamma }{1+|\bx|^2}}, \quad \bx \in \R^2,
\end{align*}
and
\begin{align*}
    \frac{1}{2\pi} \int_{\R^2}  K e^{2u}= \beta.
\end{align*}
Now, define $$\psi(\bx) := 2u(\bx)- \log(2\pi \beta)-n \log |\bx|, \quad \bx \in \R^2.$$ Then, $\psi$ satisfies 
\begin{equation}
\begin{aligned}
\label{eq:sphereicalonsager}
   & -\Delta \psi(\bx) =4\pi \beta\, |\bx|^n K(\bx) e^{\psi(\bx)},\quad \bx \in \R^2,\\
   &  \int_{\R^2} |\bx|^n K(\bx) e^{\psi(\bx)} \, \dd \bx =1.
\end{aligned}
\end{equation}
Now, by Corollary \ref{cor:existencepositiveassymdecay}
and Theorem \ref{thm:uniquenessposistivbeta}, we have
\begin{corollary}
    There exists a unique radially symmetric solution $\psi \in C^{\infty}(\R^2)$ to equation \eqref{eq:sphereicalonsager}  if $n+2+2l<\beta <n+2$.
\end{corollary}
\begin{remark}
    Regarding the necessity of the condition on $\beta$, we have the following. Since $K$ is radially decreasing, by Corollary \ref{cor:existencepositiveassymdecay}, we must have $\beta< n+2$ if $\psi \in C^{\infty}(\R^2)$ is a solutions to \eqref{eq:sphereicalonsager}, which satisfies 
\begin{align*}
   \lim_{\bx \to \infty} \frac{\psi(\bx)}{\log |\bx|} = -2\beta.
\end{align*}
Note that, by the condition $\int_{\R^2} |\bx|^n K(\bx) e^{\psi(\bx)} \, \dd \bx =1$, we have $2\beta \geq n+2+2l.$ Now, assume that $2\beta>n+2+2l.$ Then, by \eqref{eq:indexformula}, we get
\begin{align*}
    \beta -2-n &= \int_{\R^2} |\bx|^n e^{\psi(\bx)} \bx \cdot \nabla K(\bx) \, \dd \bx
    \\ &= 2l\int_{\R^2} |\bx|^{n+2} (1+|\bx|^2)^{l-1}  e^{\frac{2 \gamma }{1+|\bx|^2}}  K(\bx) e^{\psi(\bx)}  \, \dd \bx\\& -4\gamma \int_{\R^2} |\bx|^{n+2} (1+|\bx|^2)^{l-2}  e^{\frac{2 \gamma }{1+|\bx|^2}}  K(\bx) e^{\psi(\bx)}  \, \dd \bx
    \\ &= 2l +  \int_{\R^2} |\bx|^{n} (-2l(1+|\bx|^2)-4\gamma |\bx|^2) (1+|\bx|^2)^{l-2}  e^{\frac{2 \gamma }{1+|\bx|^2}}  K(\bx) e^{\psi(\bx)}  \, \dd \bx.
\end{align*}
Hence, for $\gamma \leq - \frac{l}{2}$, we obtain 
\begin{align*}
    \beta > n+2+2l.
\end{align*}
For the cases of $\gamma > - \frac{l}{2}, 2\beta>n+2+2l$ and $2\beta= n+2+2l$, we do not know if the condition $\beta >n+2+2l$ is necessary for the existence. 
   
\end{remark}

\subsection{Quantum mechanics}
\label{sec:Quantum mechanincs}
In Quantum mechanics, the $2$ dimensional trapped particles may have exotic statistical behavior which is neither symmetric like bosons nor antisymmetric like fermions; see \cite{LM-77,GMS-80,GMS-81,W-82a,W-82b,L-24}. These quasiparticles are called anyons and have been detected in experiments; see \cite{Betall-20,NLGM-20}. The proposed effective model for a large number of anyons whose statistical behavior is close to bosons and their exchanges commute with each other, named almost bosonic abelian anyons, is the Chern-Simons-Schr\"odinger-Ginzburg-Landau model, which relies on the
minimizers of the energy functional
\begin{align*}
    \cE_{\beta,\gamma,B,V}[u] = \int_{\R^2} \left|\left(\nabla  + i \left(\beta \bA[|u|^2]+  \frac{B \bx^{\perp}}{2} \right)\right)u\right|^2 -\gamma \int_{\R^2} |u|^4+ \int_{\R^2} V |u|^2,
\end{align*}
over $u \in H^1(\R^2;\C)$ with the constraints that $\int_{\R^2} |u|^2=1, \int_{\R^2} |V| \, |u|^2 < \infty$, where we assume $V \in L^1_{\loc}(\R^2), B,\gamma \in \R$, and
\begin{align*}
    \bA[|u|^2](\bx) = \int_{\R^2} \frac{(\bx-\by)^{\perp}}{|\bx-\by|^2} |u|^2(\by) \, \dd \by,
\end{align*}
for every $\bx \in \R^2.$ Here, we use the notation $\bx^{\perp} :=(-x_2,x_1)$ for $\bx=(x_1,x_2) \in \R^2$. To explain each parameter, $V$ is the external potential often $V(\bx)=|\bx|^2$ in the experiments, $B$ is the external magnetic field, $\gamma$ is the strength of attraction, where $\gamma<0$ for repulsive case and $\gamma >0$ for the attractive case, and $\beta$ is the total flux of the exchange factor of particles, which determines the exchange statistics of the particles. This model is used in the effective field theory of anyons and the fractional quantum hall effect; see \cite{ALN-24,Dun-95,Dun-99,EHI-91,EHI-91(2),HY-98,HZ-09,IL-92,JW-90,JP-92,CLR-17,CLR-18,RS-20,Z-92,T-08,ZHK-89}. Now, the ground states (the minimizes) $u$ of the functional $\cE_{\beta,\gamma,B,V}$ gives the collective probability density $|u|^2$ of a large number of anyons with the statistical parameter close to the bosons. To approximate the ground states, we consider $V=0$ and, since $\cE_{\beta,\gamma,B,V}[\overline{u}] = \cE_{-\beta,\gamma,-B,V}[u]$, it is sufficient to focus only on the case of $B \geq 0$. We call the minimizers the nonlinear Landau levels of the energy functional. The case $B=0$ has been studied in detail in \cite{ALN-24}. We mention briefly that for $\gamma = 2\pi \beta$ and $V=0$, there exists a minimizer to $  \cE_{\beta,\gamma,B=0}$ if and only if $\beta \in 2 \N$, and there is an explicit formula for all minimizers in \cite[Thm. 2]{ALN-24}. Hence, we focus only on $B >0$. Define the superpotential functional $\Phi$, in the principal value sense, by  \begin{align*}
    \Phi[\rho](\bx) := \int_{\R^2} \left(\log|\bx-\by| - \log(|\by|+1)\right) \rho(\by) \,\dd \by,
\end{align*}
    for every $\bx \in \R^2$ and $\rho \in L^1(\R^2)$. By Aharanov-Casher factorization; see \cite{AC-79} and \cite[Lem. 3.12]{ALN-24}, we derive that
\begin{align*}
    \cE_{\beta,2\pi \beta,B, V=0}[u]-  B \int_{\R^2} |u|^2 = \int_{\R^2} \left|(\partial_1 - {\rm i} \partial_2) \left( e^{\left(\beta \Phi[|u|^2]+ \frac{B |x|^2}{4}\right)} u \right) \right|^2 e^{-2 \left(\beta \Phi[|u|^2]+\frac{B |x|^2}{4}\right)},
\end{align*}
for every $u \in H^1(\R^2;\C)$, where we used the notation $\nabla = (\partial_1,\partial_2)$ and $$\curl\left(\beta \bA[|u|^2]+  \frac{B \bx^{\perp}}{2}  \right) = 2\pi \beta |u|^2 + B.$$ Hence, $u$ is a minimizer of $\cE_{\beta,2\pi \beta,B}[u]$ if $\int_{\R^2} |u|^2 =1$ and
\begin{align*}
    u(z) = \overline{f(z)} e^{\left(-\beta \Phi[|u|^2]- \frac{B |z|^2}{4}\right)},
\end{align*}
for $z \in \C$, where $f$ is a complex analytic function. Letting 
$\psi := -2\beta \Phi[|u|^2]$, we derive that 
\begin{align*}
    -\Delta \psi = 4 \pi \beta |f|^2 e^{\left(- \frac{B |z|^2}{4}\right)} e^{\psi},
\end{align*}
and 
\begin{align*}
    \int_{\R^2} |f|^2 e^{\left(- \frac{B |z|^2}{4}\right)} e^{\psi} = \int_{\R^2} |u|^2 =1.
\end{align*}
Now, if $f(z) = z^n$ for non-negative integer $n$ then, setting $\psi_B$ as $\psi$, we get
\begin{equation}
\label{eq:chernsimonsschron}
\begin{aligned}
  &  -\Delta \psi_B(\bx) =  4 \pi \beta |\bx|^{2n} e^{\left(- \frac{B |\bx|^2}{2}\right)} e^{\psi_B(\bx)}, \quad \bx \in \R^2,\\
   &  \int_{\R^2} |\bx|^{2n} e^{\left(- \frac{B |\bx|^2}{4}\right)} e^{\psi_B(\bx)} \, \dd \bx=1,
\end{aligned}
\end{equation}
where $n$ is a non-negative integer. Note that, by Gagliardo-Nirenberg interpolation inequality, we have $u \in L^{p}(\R^2)$ for every $p \geq 2$. Hence, by \cite[Lem. 3.15]{ALN-24}, we obtain
\begin{align}
\label{eq:keyassymlimit}
    \lim_{\bx \to \infty} \frac{\psi_B(\bx)}{\log |\bx|} = -2\beta.
\end{align}
By Corollary \ref{cor:existencepositiveassymdecay} and Theorem \ref{thm:uniquenessposistivbeta}, we derive the following immediate result.
\begin{corollary}
    If $2n >\beta -2$, then there exists a unique radially symmetric solution to \eqref{eq:chernsimonsschron}. Moreover, it is necessary that $2n >\beta -2$ holds to obtain a solution to \eqref{eq:chernsimonsschron}, which satisfies \eqref{eq:keyassymlimit}.
\end{corollary}
 As far as the author knows, the existence for \eqref{eq:chernsimonsschron} has only been proved for $\beta < 2$; see \cite[Thm. 2.6]{SY-92}. 

Now, it is interesting to consider the behavior of solutions as $B$ converges to $0$ and $\infty$, which, correspond to weak and strong external magnetic fields, respectively. Let $\psi$ be the radially symmetric solution of \eqref{eq:chernsimonsschron} for $B =1.$ Then, by Theorem \ref{thm:uniquenessposistivbeta}, we derive that
\begin{equation}
\label{eq:connectionofdifferentB}
\begin{aligned}
    &\psi_B(\bx) = \psi\left(\sqrt{B} \, \bx \right) + (n+1) \log(B),\\
  -2\beta \log&( |\bx| +1) - C \leq  \psi(\bx) \leq -2\beta \log( |\bx| +1) + C,
\end{aligned}
\end{equation}
 for every $\bx \in \R^2$, where $C$ is a constant depending on $\beta,n$. Hence,
\begin{align*}
    \lim_{B \to 0}\psi_B(\bx)= -\infty, \quad \textup{for every } \bx \in \R^2,
\end{align*}
  and, for every bounded open subset $\Omega \subset \R^2$, we have
\begin{align}
    \lim_{B \to 0} \int_{\Omega} |u_B|^2 = \lim_{B \to 0}   \int_{\Omega} |\bx|^{2n} e^{\left( \psi_B(\bx)- \frac{B |\bx|^2}{2}\right)} \, \dd \bx =0.
\end{align}

This implies that the local density of nonlinear Landau levels of the Chern-Simons-Schr\"odinger model converges to zero as $B \to 0$.
In particular, we do not converge to the nonlinear Landau level of the Chern-Simons-Schr\"odinger model for $B=0$, even if $\beta \in 2\N$. On the other hand, if $B \to \infty$, then, by \eqref{eq:connectionofdifferentB}, there are three cases:\\\\
(i) If $n+1>\beta$, then
\begin{align*}
   &\lim_{B \to \infty} \psi_B(\bx) = \infty, \quad \textup{for every } \bx \in \R^2,\\
    &  \lim_{B \to \infty} |\bx|^{2n} e^{\left(- \frac{B |\bx|^2}{4}\right)} e^{\psi_B} =\delta_0, \quad \textup{in the sense of measures}.
\end{align*}
(ii) If $n+1=\beta$, then
\begin{align*}
  & \lim_{B \to \infty} \psi_B(\bx) = -2\beta \log |\bx|, \quad \textup{for every } \bx \in \R^2 \setminus 0,\\
   & \lim_{B \to \infty} \psi_B(0) = \infty,\\
 &  \lim_{B \to \infty} |\bx|^{2n} e^{\left(- \frac{B |\bx|^2}{4}\right)} e^{\psi_B} =\delta_0, \quad \textup{in the sense of measures}.
\end{align*}
(iii) If $n+1<\beta$, then
\begin{align*}
  & \lim_{B \to \infty} \psi_B(\bx) = -\infty, \quad \textup{for every } \bx \in \R^2 \setminus 0,\\
   & \lim_{B \to \infty} \psi_B(0) = \infty,\\
&   \lim_{B \to \infty} |\bx|^{2n} e^{\left(- \frac{B |\bx|^2}{4}\right)} e^{\psi_B} =\delta_0, \quad \textup{in the sense of measures}.
\end{align*}

\end{document}